\documentclass[11pt]{amsart}

\usepackage{amsmath}
\usepackage{amsthm}
\usepackage{amscd}
\usepackage{amssymb}
\usepackage{graphicx}
\usepackage{latexsym} 

\newcommand{\cal}{\mathcal}

\newcommand{\chop}{\dagger}

\def\epsilon{\varepsilon}
\def\phi{\varphi}
\def\Pr{\mathbb P}
\def\hat{\widehat}

\newcommand{\supp}{\mbox{supp}}
\newcommand{\Curr}{\mbox{Curr}}
\newcommand{\Out}{\mbox{Out}}
\newcommand{\Aut}{\mbox{Aut}}

\newcommand{\BBT}{\mbox{BBT}}
\newcommand{\Ax}{\mbox{Ax}}

\newcommand{\vol}{\mbox{vol}}
\newcommand{\diag}{\mbox{diag}}
\newcommand{\diagclos}{\overline{\mbox{diag}}}
\newcommand{\Mod}{\mbox{Mod}}

\newcommand{\ILT}{\mbox{ILT}}

\newcommand{\FN}{F_N}   
\newcommand{\cvn}{\mbox{cv}_N}
\newcommand{\cvnbar}{\overline{\mbox{cv}}_N}
\newcommand{\CVN}{\mbox{CV}_N}
\newcommand{\CVNbar}{\overline{\mbox{CV}}_N}

\newcommand{\PCurr}{\Pr\Curr(\FN)}
\newcommand{\curr}{\Curr(\FN)}

\newcommand{\R}{\mathbb R}
\newcommand{\Z}{\mathbb Z}

\newcommand{\N}{\mathbb N}

\def\strutdepth{\dp\strutbox}
\def \ss{\strut\vadjust{\kern-\strutdepth \sss}}
\def \sss{\vtop to \strutdepth{
\baselineskip\strutdepth\vss\llap{$\diamondsuit\;\;$}\null}}

\def\strutdepth{\dp\strutbox}
\def \sst{\strut\vadjust{\kern-\strutdepth \ssss}}
\def \ssss{\vtop to \strutdepth{
\baselineskip\strutdepth\vss\llap{$\spadesuit\;\;$}\null}}

\def\strutdepth{\dp\strutbox}
\def \ssh{\strut\vadjust{\kern-\strutdepth \sssh}}
\def \sssh{\vtop to \strutdepth{
\baselineskip\strutdepth\vss\llap{$\heartsuit\;\;$}\null}}

\def\qed{\hfill\rlap{$\sqcup$}$\sqcap$\par}
\def\bar{\overline}
\def\tilde{\widetilde}

\def\strutdepth{\dp\strutbox}
\def \ss{\strut\vadjust{\kern-\strutdepth \sss}}
\def \sss{\vtop to \strutdepth{
\baselineskip\strutdepth\vss\llap{$\diamondsuit\;\;$}\null}}

\def\strutdepth{\dp\strutbox}
\def \sst{\strut\vadjust{\kern-\strutdepth \ssss}}
\def \ssss{\vtop to \strutdepth{
\baselineskip\strutdepth\vss\llap{$\spadesuit\;\;$}\null}}

\def\qed{\hfill\rlap{$\sqcup$}$\sqcap$\par}

\newtheorem{thm}{Theorem}[section]
\newtheorem{cor}[thm]{Corollary}
\newtheorem{lem}[thm]{Lemma}
\newtheorem{prop}[thm]{Proposition}

\newtheorem{hyp}[thm]{Hypothesis}
\theoremstyle{definition}
\newtheorem{defn}[thm]{Definition}

\newtheorem{example}[thm]{Example}

\newtheorem{facts}[thm]{Facts}
\newtheorem{convention}[thm]{Convention}
\newtheorem{convention-defn}[thm]{Convention-Definition}
\newtheorem{defn-rem}[thm]{Definition-Remark}
\newtheorem{rem}[thm]{Remark}

\newtheorem{notation}[thm]{Notation}
\numberwithin{equation}{section}

\begin{document} 
 
\author[Ilya Kapovich]{Ilya Kapovich} 
 
\address{\tt Department of Mathematics, University of Illinois at 
 Urbana-Champaign, 1409 West Green Street, Urbana, IL 61801, USA 
 \newline http://www.math.uiuc.edu/\~{}kapovich/} \email{\tt 
 kapovich@math.uiuc.edu} 
 
 \author[Martin Lustig]{Martin Lustig} 
 \address{\tt 
LATP,
Centre de Math\'ematiques et Informatique, 
Aix-Marseille Universit\'e, 
39, rue F.~Joliot Curie, 
13453 Marseille 13, 
France}
 \email{\tt Martin.Lustig@univ-amu.fr}

\thanks{The first author was supported by the NSF grant DMS-0904200}
 
\title[Invariant laminations for iwips]{Invariant laminations for irreducible automorphisms of free groups}

\begin{abstract} 
For every 
irreducible
hyperbolic 
automorphism $\phi$ of $\FN$ (i.e. the analogue of a pseudo-Anosov mapping class) it is shown that the algebraic lamination dual to the forward limit tree $T_+(\phi)$ is obtained as ``diagonal closure'' of the support of the backward limit current $\mu_-(\phi)$. This diagonal closure is obtained through a finite procedure in analogy to adding diagonal leaves from the complementary components to the stable lamination of a pseudo-Anosov homeomorphism. We also give several new characterizations as well as a structure theorem for the dual lamination of $T_+(\phi)$, in terms of Bestvina-Feighn-Handel's ``stable lamination'' associated to $\phi$.
\end{abstract}

 

 
\maketitle

\section{Introduction}\label{intro} 

For a closed surface $\Sigma$ with Euler characteristic $\chi(\Sigma) < 0$ geodesic laminations 
play an important theoretical role, for many purposes. In particular, if such a lamination $\frak L$ is equipped with a transverse measure $\mu$, it becomes a powerful tool, for example in the analysis of the mapping class $[h]$ of a homeomorphism $h: \Sigma \to \Sigma$. The set of projective classes $[\frak L, \mu]$ of such {\em measured laminations} carries a natural topology, and the issuing space $\cal{PML}(\Sigma)$ of projective measured laminations is homeomorphic to a high dimensional sphere: One of Thurston's fundamental results shows that $\cal{PML}(\Sigma)$ serves naturally as boundary to the Teich\-m\"uller space $\cal T(\Sigma)$ of $\Sigma$, and the action of the mapping class group $\Mod(\Sigma)$ extends canonically to the compact union $\cal T(\Sigma) \cup \cal{PML}(\Sigma)$.

\smallskip

In this paper we are interested in the ``cousin world'', where $\Sigma$ (or rather $\pi_1 \Sigma$) is replaced by a free group $\FN$ of finite rank $N \geq 2$, the mapping class group $\Mod(\Sigma)$ is replaced by the outer automorphism group $\Out(\FN)$, and the role of Teichm\"uller space $\cal T(\Sigma)$ is (usually) taken on by Outer space $\CVN$. These and the other terms used in this introduction will be explained in section \ref{new-background}.

\smallskip

In the $\Out(\FN)$-setting, both, topological laminations $\frak L$ and measured laminations $(\frak L, \mu)$ have natural analogues, but the situation is quite a bit more intricate:  There are two competing analogues of measured laminations, namely $\R$-trees $T$ with isometric $\FN$-action, and currents $\mu$ over $\FN$. Both, the space of such $\R$-trees as well as the space of such currents have projectivizations which can be used to compactify 
$\CVN$ (in two essentially different ways, see \cite{KL1}), and both are used as important tools to analyze single automorphisms of $\FN$.  Also, both, an $\R$-tree $T$ and a current $\mu$, determine an {\em algebraic lamination}, denoted $L^2(T)$ (the {\em dual} 
or {\em zero} lamination of $T$) and $L^2(\mu) = \supp(\mu)$ (the {\em support} of $\mu$) respectively. Algebraic laminations $L^2$
for a free group $\FN$ are the natural analogues of (non-measured) geodesic laminations $\frak L$ on a surface $\Sigma$ 
(or rather, of their lift $\tilde{\frak{L}} \subset \tilde \Sigma$ to the universal covering $\tilde \Sigma$ of $\Sigma$: the set $L^2$ consists of $\FN$-orbits of pairs $(X, Y) \in \partial^2\FN := \partial\FN \times \partial\FN \smallsetminus \{(Z, Z) \mid Z \in \partial\FN \}$).

\smallskip

The two spaces which serve (after projectivization) as analogues of $\cal T(\Sigma) \cup \cal{PML}(\Sigma)$, namely the space $\cvnbar$ of very small $\R$-tree actions of $\FN$, and the space $\curr$ of currents over $\FN$, are related naturally by an $\Out(\FN)$-equivariant continuous {\em intersection form} $\langle \cdot, \cdot \rangle: \cvnbar \times\curr \to \R_{\geq 0}$ (see \cite{KL2}). In \cite{KL3} it has been shown that $\langle T, \mu\rangle = 0$ if and only if 
$\supp(\mu) \subset L^2(T)$. Sometimes this inclusion is actually an equality; 
in particular, this can happen if the $\FN$-action on $T$ is free, but even for free actions it isn't automatic: A more subtle condition is needed, which, in the setting of a closed surface $\Sigma$, with $T = T_\mu$ dual to 
a measured lamination $(\frak L, \mu)$ on $\Sigma$, would correspond to
the condition that
$\frak L$ is {\em maximal} (i.e. every complementary component has 
3 cusps, and at the same time 
{\em minimal} (i.e. no proper sublamination): 
then the current $\hat \mu$ defined by $(\frak L, \mu)$ satisfies $\supp(\hat\mu) = L^2(T_\mu)$.

\smallskip

However, if some complementary component of $\frak L$ has more than 3 cusps, then $L^2(T_\mu)$ will always contain all diagonal leaves of that component, while those are missing in $\supp(\hat\mu)$.
This is the basic phenomenon why
in general one can not deduce from $\langle T, \mu\rangle = 0$ the equality $\supp(\mu) = L^2(T)$: The dual laminations of $\R$-trees are by nature always {\em diagonally closed}, i.e. the {\em diagonal closure} $\diagclos(L^2(T))$ 
(see subsection \ref{laminations}) 
is equal to $L^2(T)$ for all $T \in \cvnbar$. On the other hand, the support of a current may well be minimal and thus typically a proper subset of its diagonal closure. 

\smallskip

Moreover, it is easy to find {\em perpendicular} pairs $(T, \mu)$ (i.e. $T$ and $\mu$ satisfy $\langle T, \mu\rangle = 0$) where even the diagonal closure $\diagclos(\supp(\mu))$ is only a proper subset of $L^2(T)$. 
A class of examples with even stronger properties is described in section \ref{discussion}; alternatively, any measured lamination $(\frak L, \mu)$ on a surface which has a complementary component with non-abelian $\pi_1$  defines a perpendicular pair $(T_\mu, \hat \mu)$ as above, with $\diagclos(\supp(\hat\mu)) \neq L^2(T_\mu)$.

\smallskip

For pseudo-Anosov mapping classes $[h]$ of $\Sigma$ it is well known that the induced action on $\cal T(\Sigma) \cup \cal{PLM}(\Sigma)$ has uniform North-South dynamics.  In particular, $[h]$ has precisely one attracting and one repelling fixed point on $\cal{PLM}(\Sigma)$.  These very statements are also true for the induced actions of {\em atoroidal iwip automorphisms} of $\FN$ 
(as defined in subsection \ref{iwips})
on both ``cousin spaces'', $\CVNbar$ and $\PCurr$, so that the (much studied) class of atoroidal iwip automorphisms should be considered as strict analogue of the class of pseudo-Anosov mapping classes. 
In subsection \ref{iwips} we explain that this class also coincides with the class of {\em irreducible hyperbolic automorphisms} of $\FN$.
In \cite{KL3} it has been shown that for any atoroidal iwip $\phi \in \Out(\FN)$ the {\em forward limit tree} $T_+ = T_+(\phi) \in \cvnbar$, 
which defines the attracting fixed point $[T_+] \in \CVNbar$, 
is perpendicular to the {\em backward limit current} $\mu_- = \mu_-(\phi) \in \curr$,  
which gives the repelling fixed point $[\mu_-] \in \PCurr$:
$$\langle T_+(\phi), \mu_-(\phi)\rangle = 0 \quad {\rm for \,\, any\,\, atoroidal\,\, iwip} \quad \phi \in \Out(\FN)$$

The main result of this paper can now be stated as follows:

\begin{thm}
\label{main-thm}
Let $\phi$ be an atoroidal iwip automorphism of $\FN$. Let $T_+$ be its forward limit tree, and let $\mu_-$ be its backward limit current.
Then the dual lamination $L^2(T_+)$ and the diagonal closure of the support $\supp (\mu_-)$ satisfy:
$$L^2(T_+) = \diagclos(\supp (\mu_-)) = \diag(\supp (\mu_-))$$
\end{thm}

(Note here that the {\em diagonal extension} $\diag(L^2))$ of an algebraic lamination $L^2$, defined in subsection \ref{laminations}, is always contained in the diagonal closure $\diagclos(L^2)$, but in general it is more practical to handle than the latter.)

\smallskip

This result raises the question, what the precise conditions are, which allow one to deduce for perpendicular pairs $(T, \mu) \in \cvnbar \times \curr$ the analogous conclusion, i.e. that $L^2(T) = \diagclos(\supp (\mu))$. A necessary condition for such a conclusion is that $\mu$ ``fills out'' enough of the available room in $\FN$ which is ``potentially dual'' to $T$.  Making such an indication precise takes more room than available here in the introduction.  We will discuss those matters below in section \ref{discussion}.

\smallskip

A second result of this paper, which is stronger than Theorem \ref{main-thm}, concerns the structure of the lamination $L^2(T_+)$. The proof of this Theorem \ref{main-result2} is really the main purpose of this paper; we show below how to derive Theorem \ref{main-thm} from 
it.  Theorem \ref{main-result2} is also the crucial ingredient in our follow-up paper 
\cite{KL-CT}
on the fibers of the Cannon-Thurston map for iwip automorphisms of $\FN$.

\begin{thm}
\label{main-result2}
Let $\phi$ be an atoroidal iwip automorphism of $\FN$, and let $T_+$ be its forward limit tree. Then there exists a sublamination $L^2_{BFH} \subset L^2(T_+)$ which satisfies:
\begin{enumerate}
\item
$L^2_{BFH}$ is the ``stable lamination'' for $\phi^{-1}$, as exhibited by Bestvina-Feighn-Handel \cite{BFH}. In particular, $L^2_{BFH}$ is minimal and non-empty.
\item
$L^2_{BFH}$ is the only minimal and non-empty sublamination of $L^2(T_+)$.
\item
$L^2(T_+) = \diagclos(L^2_{BFH}) = \diag(L^2_{BFH})$
\item
$L^2(T_+) \smallsetminus L^2_{BFH}$ is a finite union of $\FN$-orbits of pairs $(X, Y) \in \partial^2 \FN$.
\end{enumerate}
\end{thm}

The arguments given below 
prove also that $L^2_{BFH}$ is contained in $\supp(\mu_-)$. Indeed, 
a direct argument shows that the two laminations are equal, see \cite{CKL} or \cite{CP}.

\smallskip

In the course of our proof we also give several alternative characterizations of the dual lamination $L^2(T_+)$, based on (absolute) train track representatives of $\phi$ or of $\phi^{-1}$. The precise terminology of the terms used 
in the following statement 
is given in sections \ref{steps-one-and-two} and \ref{steps-three-and-four} below.

\begin{prop}
\label{alternative-characterizations}
Let $\phi$ be an atoroidal iwip automorphism of $\FN$, and let $T_+$ be its forward limit tree. Let $f_+: \tau_+ \to \tau_+$ and $f_-: \tau_- \to \tau_-$ be 
stable train track maps that represent 
suitable powers
$\phi^k$ and $\phi^{-k}$ (with $k \geq 1$) respectively.

\smallskip
\noindent
(1)
The dual lamination $L^2(T_+)$ consists precisely of all pairs $(X, Y) \in \partial^2\FN$ such that for some $C > 0$ the whole $\phi$-orbit $\{\phi^t(X,Y) \mid t \in \Z\}$ is totally $C$-illegal with respect to $f_+$.

\smallskip
\noindent
(2)
The dual lamination $L^2(T_+)$ consists precisely of all pairs $(X, Y) \in \partial^2\FN$ such that the whole $\phi$-orbit is uniformly $\phi$-contracting (or, equivalently, uniformy $\phi^{-1}$-expanding).

\smallskip
\noindent
(3)
The dual lamination $L^2(T_+)$ consists precisely of all pairs $(X, Y) \in \partial^2\FN$ such that the whole $\phi$-orbit is used legal with at most one singularity, with respect to 
$f_-$.
\end{prop}

This proposition is a direct consequence of the fact that in the course of our proof of Theorem \ref{main-result2} we show first that $L^2(T_+)$ is included in the set of pairs from $\partial^2\FN$ defined by the the property stated in part (1) of Theorem \ref{main-result2}.  One then shows that this set of pairs is contained in the set defined by part (2), and that the latter is contained in the set defined by part (3). In a final step one shows the set of $(X,Y)$ defined in part (3) is contained in the diagonal closure of $L^2_{BFH}$. 
But by part (3) of Theorem \ref{main-result2} the lamination $L^2(T_+)$ is equal to this diagonal closure, so that all of the previous inclusions must be in fact equalities.

\smallskip

Using the same type of arguments one deduces Theorem \ref{main-thm} from Theorem \ref{main-result2}:  From the perpendicularity of $T_+$ with $\mu_-$ we derived above that $\supp(\mu_-)$ is a subset of $L^2(T_+)$, which implies (compare Definition-Remark \ref{diagonal-closure}) $\diagclos(\supp(\mu_-)) \subset \diagclos(L^2(T_+))$.  Since $L^2(T_+)$ is diagonally closed (see Proposition \ref{dual-diag-closed}), we obtain:
$$\diagclos(\supp(\mu_-)) \subset L^2(T_+)$$
On the other hand, from the uniqueness property given by part (2) of Theorem \ref{main-result2} one deduces that $L^2_{FBH}$ is contained in $\supp(\mu_-)$. Hence (again by Definition-Remark \ref{diagonal-closure}) $\diag(L^2_{FBH})$ must be contained in $\diag(\supp(\mu_-)) \subset \diagclos(\supp(\mu_-))$, so that part (3) of Theorem \ref{main-result2} yields
$$L^2(T_+) \subset \diag(\supp(\mu_-)) \subset \diagclos(\supp(\mu_-)) \, .$$

\medskip
\noindent
{\em Acknowledgements:}
This paper has been influenced by many discussions with colleagues in our scientific neighborhood. In particular we would like to thank T. Coulbois and A. Hilion, and also point to their forthcoming joint work\footnote{\, In the mean time this has been written up, see \cite{CHR}}
with P. Reynolds, 
where much related (but harder) questions are treated with different (and distinctly deeper) methods.

\medskip
\noindent
{\em Historical comment:}
The main result of this paper was originally referred to under the name of ``Seven steps to happiness'' 
(compare \cite{KL-7steps} and \cite{Lu3}), due to the very origin of this paper, where the more optimistic of the two authors proposed a proof-scheme in seven steps, while the other, the more skeptical between us, decided that this reminded him of one of those advertised programs which guarantee personal fulfillment, wealth, or religious enlightenment. Since then, the result has been internally 
referred to 
by that title, and in particular the headers of the later chapters are still a reminiscence to the beginning of this particular collaboration.

\bigskip

\noindent
\underline{Organisation of the paper:}
\begin{itemize}
\item
In \S \ref{new-background} we recall the basic facts and definitions about $\R$-trees, Outer space and iwip automorphisms. We also review some of the basics of algebraic laminations, define the diagonal 
extension and 
closure of a lamination, and recall some of the known facts about the dual lamination of an $\R$-tree.
\item
In \S \ref{train-track-section} we recall the basic facts and definitions of train track maps. We purposefully employ space and care to make this section accessible for, say, a graduate student who is only partially an expert on the subject.
\item
In \S \ref{steps-one-and-two} we show that the dual lamination $L^2(T_+)$ of the forward limit tree $T_+$ of an atoroidal iwip automorphism $\phi$ has a strong contraction property with respect to the action of $\phi$. This implies that $L^2(T_+)$ has a strong expansion property with respect to the action of $\phi^{-1}$
\item
In \S \ref{steps-three-and-four} we will show that any lamination with this strong expansion property with respect to $\phi^{-1}$, when realized as reduced paths on a train track representative $f_-: \tau_- \to \tau_-$ of $\phi^{-1}$, consists only of paths which are entirely ``used legal'', or else have precisely one ``singularity'' (the terms will be defined there).
It will also be shown that the sublamination of $L^2(T_+)$ which is realized by used legal paths on $\tau_-$ is 
contained in Bestvina-Feighn-Handel's 
``stable'' 
lamination $L^2_{BFH}(f_-)$ associated to the train track map $f_-$.
\item
In \S \ref{steps-five-to-seven} we consider elements $(X, Y) \in L^2(T_+)$ which are realized on $\tau_-$ by used legal paths with precisely one singularity. We show that such paths have a rather special form: they are (essentially) the concatenation of two eigenrays of the map $f_-$. 
As a consequence, we show that such pairs $(X,Y)$ must lie in the diagonal extension of $L^2_{BFH}(f_-)$. 
Also, the finiteness of such pairs (up to the $\FN$-action) 
is a direct consequence of this characterization.
\item
In \S \ref{discussion} we discuss some question issuing from our main result and the surrounding facts.
\end{itemize}

\section{Definition of terms and background}
\label{new-background}

The purpose of this section is to properly define the terms used in the introduction, and to give some background with references about them.

Throughout the paper we fix a ``model'' free group $\FN$ of finite rank $N \geq 2$.  
A finite connected graph $\tau$ is called a {\em marked graph}, if it is equipped with a {\em marking isomorphism} $\theta: \FN \overset{\cong}{\longrightarrow} \pi_1(\tau)$.  Here we purposefully suppress the issue of choosing a basepoint of $\tau$, as in this paper we are interested in automorphims of $\FN$ only up to inner automorphisms.

\medskip

\subsection{Outer space} 
\label{Outer-space}

${}^{}$
\smallskip

We give here only a brief overview of basic facts related to Outer space. We refer the reader to~\cite{CV}
for more detailed background information.

The \emph{unprojectivized Outer space} $\cvn$ consists of all minimal free and discrete isometric actions on $F_N$ on $\mathbb R$-trees (where two such actions are considered equal if there exists an $F_N$-equivariant isometry between the corresponding trees). There are several different topologies on $\cvn$ that are known to coincide, in particular the equivariant Gromov-Hausdorff convergence topology and the so-called 
\emph{length function} topology. Every $T\in \cvn$ is uniquely determined by its \emph{translation length function} $||\cdot||_T:F_N\to\mathbb R$, where $||g||_T:= \min\{d(gx,x) \mid x \in T\}$ is the {\em translation length} of $g$ on $T$. 
Indeed, $||g||_T = 0$ implies that $g$ has a fixed point in $T$, while for $||g||_T > 0$ there is an {\em axis} $\Ax(g)$ in $T$ which is isometric to $\R$ and along which $g$ translates by the amount $||g||_T $. 

Two trees $T_1,T_2\in\cvn$ are close if the functions $||\cdot||_{T_1}$ and $||\cdot||_{T_1}$ are close pointwise on a large ball in $F_N$. The closure $\cvnbar$ of $\cvn$ in either of these two topologies is well-understood and known to consists precisely of all the so-called \emph{very small} minimal isometric actions of $F_N$ on $\mathbb R$-trees, see \cite{BF93} and \cite{CL}. 
The space $\partial \cvn := \cvnbar \smallsetminus \cvn$ is sometimes called the {\em Thurston boundary} of unprojectivized Outer space $\cvn$.

The outer automorphism group $\Out(F_N)$ has a natural continuous right action on $\cvnbar$ (that leaves $\cvn$ invariant) given at the level of length functions as follows: for $T\in \cvnbar$ and $\phi\in Out(F_N)$ we have $||g||_{T\phi}=||\Phi(g)||_T$, 
with
$g\in F_N$
and
$\Phi\in Aut(F_N)$ representing $\phi \in \Out(F_n)$.

The \emph{projectivized Outer space} $\CVN=\mathbb P\cvn$ is defined as the quotient $\cvn/\sim$ where for $T_1\sim T_2$ whenever $T_2=cT_1$ for some $c>0$. One similarly defines the projectivization $\CVNbar=\mathbb P\cvnbar$ of $\cvnbar$ as $\cvnbar/\sim$ where $\sim$ is the same as above. The space $\CVNbar$ is compact and contains $\CVN$ as a dense $\Out(F_N)$-invariant subset. The compactification $\CVNbar$ of $\CVN$ is a free group analogue of the Thurston compactification of the Teichm\"uller space. For $T\in \cvnbar$ its $\sim$-equivalence class is denoted by $[T]$, so that $[T]$ is the image of $T$ in $\CVNbar$.

\bigskip
\subsection{Laminations and the diagonal closure}
\label{laminations}

${}^{}$
\smallskip

For the free group $\FN$ we define the {\em double boundary}
\[
\partial^2 \FN:=
\{(X, Y)\in \partial F_N \times \partial F_N \mid X\ne Y\}. 
\]
The set $\partial^2 \FN$ comes equipped with a natural topology,
inherited from $\partial F_N \times \partial F_N$, and with a 
natural translation action of $F_N$ by homeomorphisms. There is also a
natural {\em flip} map $\partial^2 \FN \to \partial^2 \FN$,
$(X, Y)\mapsto (Y, X)$, interchanging the two coordinates on
$\partial^2 \FN$.

The following definition has been introduced and systematically studied in \cite{CHL1-I}, where also further background material concerning this subsection can be found.

\begin{defn}
\label{algebraic-lamination}
An \emph{algebraic lamination} on $F_N$ is a non-empty subset 
$L^2\subseteq
\partial^2 F_N$ which is (i) closed, (ii) $F_N$-invariant, and (iii) flip-invariant.

In analogy to laminations on surfaces, the elements $(X, Y) \in L^2$ are sometimes also referred to as the {\em leaves} of the algebraic lamination $L^2$.
\end{defn}

In some circumstances it is useful to admit the empty set as algebraic lamination. We will formally stick to the classical non-empty notion, but occasionally informally include the empty set in our discussions about algebraic lamination.

\begin{example}
\label{conj-classes}
Let $\Omega$ be a set of conjugacy classes $[w]$ with $w \in \FN$. Denote by $w^\infty$ and $w^{-\infty}$ the attracting and the repelling fixed points respectively of the action of $w$ on $\partial \FN$. Then
$$L^2(\tilde \Omega) := \overline{\{(v w^\infty, v w^{-\infty}) \mid v \in \FN, [w] \in \Omega \cup \Omega^{-1}\}}$$
is an algebraic lamination. 
\end{example}

\begin{rem}
\label{finite-laminations}
(a)
Note that, in the above Example \ref{conj-classes}, for finite $\Omega$ taking the closure on the right hand side of the displayed equality can be omitted without changing the set $L^2(\tilde\Omega)$.

\smallskip
\noindent
(b)
Thus for finite $\Omega$ one obtains a lamination $L^2(\tilde\Omega)$ which consists only of finitely many $\FN$-orbits of leaves. Indeed, it is an easy exercise to show that any such {\em finite} lamination occurs precisely in this way. (For the less experienced reader we recommend for this exercise the transition to 
{\em symbolic laminations}, as described in detail in \cite{CHL1-I},
see also the discussion below, after Remark \ref{non-closed}.)
\end{rem}

\begin{defn}
\label{minimal-laminations}
An algebraic lamination $L^2 \neq \emptyset$ is called {\em minimal} if it doesn't contain any proper sublamination (other than $\emptyset$). This is equivalent to requiring that for any $(X, Y) \in L^2$ the set $\FN \cdot (X, Y) \cup \FN \cdot (Y, X)$ is dense in $L^2$. 
\end{defn}

\smallskip

The following terminology is inspired by geodesic laminations on surfaces, where isolated leaves can cross diagonally through the complementary components of the lamination.

\begin{defn-rem}
\label{diagonal-closure}
Let $S$ be a subset of $\partial^2 F_N$.
\begin{enumerate}
\item[(a)]
The {\em diagonal extension} 
$\diag(S)$ of $S$ is the set of all $(X,Y)\in \partial^2 F_N$
such that for some integer $m\ge 1$ there exist elements $X=Z_0,
Z_1,\dots,$ $Z_m=Y$ in $\partial F_N$ such that
$(Z_{i-1},Z_{i})\in S$ for $i=1,\dots, m$.

\smallskip
\noindent
\item[(b)] It is easy to see that $S\subseteq \diag(S)$ and that $\diag(\diag(S))=\diag(S)$.

\smallskip
\noindent
\item[(c)]
A subset $S\subseteq \partial^2 F_N$ is said to be {\em diagonally 
extended} if $S = \diag(S)$.

\smallskip
\noindent
\item[(d)]
If $S'\subseteq S\subseteq \partial^2 F_N$ then  $\diag(S') \subseteq \diag(S)$.

\smallskip
\noindent
\item[(e)]
If $S$ is $F_N$-invariant, then so is $\diag(S)$.  Similarly, if $S$ is flip-invariant then so is $\diag(S)$. However, if $S$ is a closed subset of $\partial^2 S$, then we can deduce that $\diag(S)$ is closed only of $\diag(S) \smallsetminus S$ consists of finitely many $\FN$-orbits.

\smallskip
\noindent
\item[(f)]
We denote by $\diagclos(S)$ the closure of $\diag(S)$ in $\partial^2\FN$. For any algebraic lamination $L^2$ the {\em diagonal closure} $\diagclos(L^2)$ is again an algebraic lamination.  
We say that $L^2$ is {\em diagonally closed} if $L^2 = \diagclos(L^2)$.
If $\diag(L^2) \smallsetminus L^2$ consists of finitely many $\FN$-orbits, then $\diagclos(L^2) = \diag(L^2)$.

\end{enumerate}
\end{defn-rem}

\begin{rem}
\label{non-closed}
Examples of laminations where the diagonal extension is not closed are easy to come by. For example, 
if $L^2$ given by any minimal filling geodesic lamination 
$\frak L$ on a surface $\Sigma$ with boundary, then $\diag(L^2)$ 
contains every pair $(X, Y) \in \partial^2\FN$ given by a diagonal leaf in a complementary component
of $\frak L$ on $\Sigma$, 
but none of the rational laminations defined by the boundary curves of $\Sigma$, which however belong all to $\diagclos(L^2)$.

\end{rem}

If one picks a basis $\cal A$ for $\FN$, one can naturally associate to any pair $(X, Y) \in \FN$ a biinfinite reduced word $Z = X^{-1} \cdot Y$ in $\cal A\cup \cal A^{-1}$, and to $Z$ the set $\cal L$ of all finite subwords. In this way one can associate to every algebraic lamination $L^2$ a {\em symbolic lamination} $L_\cal A$ and a laminary language $\cal L_\cal A$, and both of these translations are invertible. For details see \cite{CHL1-I}.

Interpreting words in $\cal A^{\pm 1}$ as paths in the the associated rose $\rho_\cal A$, or in its universal covering
$\tilde \rho_\cal A$ (which is identical to the Cayley graph $\Gamma(\FN, \cal A)$), one can easily generalize the concepts of a ``symbolic lamination'' or a ``laminary language'' to more general graphs $\tau$ with an identification $\FN = \pi_1(\tau)$, or to their universal coverings. For the purposes of this paper, however, the following suffices:

\begin{defn}
\label{covering-versus-base}
Let $\tau$ be a graph with a marking $\theta: \FN \overset{\cong}{\longrightarrow} \pi_1 (\tau)$, and let $\partial \theta: \partial \FN \to \partial \tau$ be the induced homeomorphisms on the Gromov boundaries.

Let $(X, Y) \in \partial^2 \FN$, and let $\tilde \gamma$ be the biinfinite reduced path (the ``geodesic'') in the universal covering $\tilde \tau$ of $\tau$ which joins the boundary point $\partial \theta(X)$ to the boundary point $\partial \theta(Y)$. 

The reduced biinfinite path $\gamma$ in $\tau$ which is the image of $\tilde \gamma$, under the universal covering map $\tilde \tau \to \tau$, is called {\em the geodesic realization} (in $\tau$) of the pair $(X,Y)$, and is denoted by 
$\gamma_\tau(X, Y)$.
\end{defn}

Assume now that the marked graph $\tau$ comes with a homotopy equivalence $f: \tau \to \tau$ which {\em represents} 
$\phi \in \Out(\FN)$, 
i.e. 
some lift $\Phi \in \Aut(\FN)$ of $\phi$ satisfies
$\theta \circ \Phi  = f_* \circ \theta$.
In this case it follows directly from the above definition that,
for any integer $t \geq 0$ and any $(X, Y) \in \partial^2 \FN$, the geodesic realizations satisfy the equality
$$[f^t(\gamma_\tau(X, Y))] = \gamma_\tau(\partial\Phi^t(X), \partial\Phi^t(Y)) \, ,$$
where for any possibly non-reduced path $\eta$ we denote by $[\eta]$ the 
corresponding
reduced path 
(see Convention-Definition \ref{paths}).

\begin{rem}
\label{no-lift}
For distinct lifts $\Phi, \Phi' \in \Aut(\FN)$ of $\phi \in \Out(\FN)$ the lifts to $\tilde \tau$ of $\gamma_\tau(\partial\Phi(X), \partial\Phi(Y))$ and $\gamma_\tau(\partial\Phi'(X), \partial\Phi'(Y))$ differ only by the deck transformation induced by some $w \in \FN$ on $\tilde \tau$, so that one has:
$$\gamma_\tau(\partial\Phi(X), \partial\Phi(Y)) =\gamma_\tau(\partial\Phi'(X), \partial\Phi'(Y))$$
As a consequence, to smoothen the notation, we will in the sequel often simply denote this geodesic realization by
$\gamma_\tau(\phi(X),\phi(Y))$.
\end{rem}

We say that a lamination $L^2$ {\em is generated} by a set $\cal P$ of reduced edge paths $\gamma_i$ in a marked graph $\tau$, if a pair $(X, Y) \in \partial^2 \FN$ belongs to $L^2$ if and only if any finite subpath of the geometric realization $\gamma_\tau(X, Y)$  is also a subpath of some 
path $\gamma_i$ which satisfies 
$\gamma_i \in \cal P$ or $\bar \gamma_i \in \cal P$. In this case we write:
$$L^2 = L^2(\cal P)$$
Here typically the $\gamma_i$ are finite paths, and $\cal P$ is infinite, but we also include the case of closed, infinite or biinfinite $\gamma_i$.

\begin{rem}
\label{generated-contained}
It follows directly from this definition that for any algebraic lamination $L^2$ and any leaf $(X, Y) \in L^2$ with geodesic realization $\gamma := \gamma_\tau(X,Y)$ the lamination generated by $\gamma$ satisfies
$$L^2(\gamma) \subset L^2$$
(where we abbreviate $L^2(\{\gamma\})$ to $L^2(\gamma)$).
\end{rem}

It is easy to see (compare \cite{CHL1-I}) that any infinite set $\cal P$ of finite (possibly closed) reduced edge paths in $\tau$ generates a 
(non-empty) 
algebraic lamination $L^2(\cal P)$. If all $\gamma_i \in \cal P$ are closed paths, then $L^2(\cal P)$ only depends on the set $\Omega$ of conjugacy classes in $\FN$ defined by the $\gamma_i \in \cal P$ (and not on the particular choice of the marked graph $\tau$ and the geodesic realization of the conjugacy classes in $\tau$); in this case we write $L^2(\Omega) := L^2(\cal P)$.

\smallskip

We would like to point the reader's attention to the subtle difference between the laminations $L^2(\Omega)$ and the previously introduced lamination $L^2(\tilde\Omega)$ (see Example \ref{conj-classes}):  One always has $L^2(\Omega) \subset L^2(\tilde\Omega)$ and $ L^2(\tilde\Omega) = L^2(\tilde{\cal P})$, where for any infinite set $\cal P$ of closed loops in $\tau$ that determines $\Omega$ we mean by $\tilde{\cal P}$ the set of precisely those biinfinite reduced paths in $\tau$ that wind around the closed paths from $\Omega$.

\medskip
\subsection{The algebraic lamination dual to an $\R$-tree}
\label{dual-lamination}

${}^{}$
\smallskip

To any tree $T\in \cvnbar$ in \cite{CHL1-II} there has been naturally associated a
\emph{dual algebraic lamination} (also called the \emph{zero
  lamination}) $L^2(T)$ of $T$:

\begin{defn}
\label{L2}
Let $T$ be an $\R$-tree from $\cvnbar$.

\smallskip
\noindent
(1)
For any $\epsilon > 0$ let $\Omega_\epsilon(T)$ be the set of conjugacy classes 
$[w] \subset \FN \smallsetminus \{ 1 \}$ with translation length $|| w ||_T \leq \epsilon$, and let $L^2_\epsilon(T) := L^2(\Omega_\epsilon(T))$.

\smallskip
\noindent
(2) Define $L^2(T) := {\underset{\epsilon > 0}{\bigcap}} L^2_\epsilon(T)$.

\end{defn}

In \cite{CHL1-II} it has been shown that 
set $L^2(T)$ is an algebraic lamination on $F_N$, if $T \in \partial \cvn$, and $L^2(T) = \emptyset$ otherwise (i.e. $T \in \cvn$).
It has also been shown in \cite{CHL1-II}, for any basis $\cal A$ of $\FN$, that
$L^2(T)$ is precisely the set of all $(X,Y)\in \partial^2 F_N$ such that the associated reduced biinfinite word $Z = X^{-1} Y$ in $\cal A^{\pm 1}$ has the property that for every finite subword $w$ of $Z$ and any $\epsilon > 0$ there is a cyclically reduced word $v$ which contains $w$ as subword and satisfies $|| v ||_T \leq \epsilon$.

As a slight extension of the latter, we can formulate the criterion which alternatively could serve as definition for this paper:

\begin{rem}
\label{dual-lamination-criterion}
For any marked graph $\tau$ the  following characterization of $L^2(T)$ holds: 

A finite 
reduced path $\gamma'$ in $\tau$ is a subpath of the geodesic realization $\gamma_{\tau}(X,Y)$ for some $(X, Y) \in L^2(T)$ if and only if for any $\epsilon > 0$ there is an element $w \in \FN$ with translation length on $T$ of size $||w||_{T} \leq \epsilon$, such that the conjugacy class of $w$ in $\FN$ is represented by a reduced loop $\hat \gamma$ in $\tau$ which contains $\gamma'$ as subpath.
\end{rem}

\begin{rem}
\label{dual-lamination-projectivized}
We observe directly from the definition that the zero lamination of any $\R$-tree $T$ depends only on its projective class $[T]$, i.e. :
$$L^2(T) = L^2(\lambda T) \quad {\rm for \,\, any} \quad \lambda > 0 \, .$$
\end{rem}

For any $\R$-tree $T$ a point $x \in T$ is called a 
{\em branch point} if $T \smallsetminus \{x\}$ has 3 or more conected components.
It is well known 
and easy to see 
that every $T \in \partial \cvn$ either contains 
an ``edge''  (i.e. a
non-degenerate segment without branch points),
or else $T$ has {\em dense orbits}: This means that 
the
orbit $\FN x$ of 
any 
point $x \in T$ is a dense subset of $T$. Note that this means that in particular the set of 
{branch points} 
is dense in $T$.  Of course, such a tree $T \in \partial \cvn$ with dense orbits may well contain points $x$ with non-trivial point stabilizers ${\rm stab}(x) \subset \FN$.

\smallskip

For $T\in\cvnbar$ with dense $F_N$-orbits one can give an
equivalent characterization of $L^2(T)$ to that given in
Definition~\ref{L2} above. For such a tree $T$ there is
a canonical $F_N$-equivariant map $Q:\partial F_N\to \widehat T=\overline T\cup
\partial T$, constructed in~\cite{CHL1-II}, where 
$\overline T$ is the metric completion of $T$ and $\partial T$ denotes its Gromov boundary.
The precise nature of how the map $Q$ is defined is not relevant for the present paper. However, it is proved in~\cite{CHL1-II} that in this case for $(X,Y)\in
\partial^2 F_N$ we have $(X,Y)\in L^2(T)$ if and only if
$Q(X)=Q(Y)$. This fact immediately implies:

\begin{prop}
\label{dual-diag-closed}
Let $T\in\cvnbar$ have dense $F_N$-orbits. 
Then $L^2(T)$
is diagonally closed, that is $\diagclos(L^2(T))=L^2(T)$.
\qed
\end{prop}

\medskip
\subsection{Iwip automorphisms}
\label{iwips}

${}^{}$
\smallskip

\begin{defn}
\label{defn:iwips}
\noindent (a)  An outer automorphism $\phi\in \Out(F_N)$ is called {\em irreducible with irreducible powers (iwip)} if no positive power of $\phi$ preserves the conjugacy class of a
proper free factor of $F_N$.

\smallskip
\noindent
(b)
An outer automorphism  $\phi\in \Out(F_N)$ is
called \emph{atoroidal} if it has no non-trivial periodic conjugacy classes, i.e. if there do not exist $t\ge 1$ and $w\in F_N-\{1\}$ such that
$\phi^t$ fixes the conjugacy class $[w]$ of $w$ in $F_N$.
\end{defn}

It was proved in \cite{BH} that for $N\ge 2$ an iwip automorphisms $\phi \in \Out(\FN)$ is non-atoroidal if and only if $\phi$ is induced, via an identification of $\FN$ with the fundamental group of a compact  surface $S$ with a single boundary component, by a pseudo-Anosov homeomorphisms $h: S \to S$.

\begin{rem}\label{defn:irr}
(1)
The terminology ``iwip'' derives from the groundbreaking paper \cite{BH}: Bestvina-Handel call
an element $\phi\in \Out(F_N)$ is \emph{reducible} if there
exists a 
non-trivial 
free product decomposition $F_N=C_1\ast\dots C_k\ast F'$,
where $k\ge 1$ and $C_i\ne \{1\}$, such that $\phi$ permutes the
conjugacy classes of subgroups $C_1,\dots, C_k$ in $F_N$. An element
$\phi\in \Out(F_N)$ is called \emph{irreducible} if it is not
reducible.

\smallskip
\noindent
(2)
It is not hard to see that an element $\phi\in \Out(F_N)$ is an
iwip if and only if for every $n\ge 1$ the power $\phi^n$ is irreducible (sometimes such automorphisms are also called \emph{fully irreducible}).

\smallskip
\noindent
(3)
An automorphisms $\psi$ of $\FN$ is called {\em hyperbolic} if its mapping torus group $\FN\rtimes_\psi \Z$ 
is word-hyperbolic.  It was proved by Brinkmann \cite{Br} that $\psi$ is hyperbolic if and only if no non-trivial conjugacy class of $\FN$ is fixed by a positive power of $\psi$. It has been shown in 
\cite{DKL} (see also \cite{Ka2013}) that for hyperbolic automorphisms the notion of irreducibility and full irreducibility are equivalent, so that $\phi \in \Out(\FN)$ is an atoroidal iwip if and only if $\phi$ is an irreducible hyperbolic automorphism.
\end{rem}

It is known by a result of 
Levitt-Lustig~\cite{LL} that
iwips have a simple ``North-South''
dynamics on the compactified Outer space $\CVNbar$:

\begin{prop}\label{prop:LL}\cite{LL}
Let $\phi\in \Out(F_N)$ be an iwip. Then there exist unique $[T_+]=[T_+(\phi)],[T_-]=[T_-(\phi)]\in \CVNbar$ with the following properties:
\begin{enumerate}
\item The elements $[T_+],[T_-]\in \CVNbar$ are the only fixed points of $\phi$ in $\CVNbar$.
\item For any $[T]\in \CVNbar$, $[T]\ne [T_-]$ we have $\lim_{n\to\infty} [T\phi^n]=[T_+]$ and for any $[T]\in \CVNbar$, $[T]\ne [T_+]$ we have $\lim_{n\to\infty} [T\phi^{-n}]=[T_-]$.

\item 
We have $T_+\phi=\lambda_+T$ and $T_-\phi^{-1}=\lambda_-T_-$   where $\lambda_+>1$ and $\lambda_->1$. 

\item
Both, $T_+$ and $T_-$ are ``intrinsically non-simplicial'' $\R$-trees:  Every $\FN$-orbit of points is dense in the tree.

\end{enumerate}
\end{prop}

In~\cite{LL} it is also proved that convergence in (2) is locally
uniform and hence uniform on compact subsets. 
For more information see \cite{KL6}.
More details about 
$T_+$ in terms of train tracks 
are given below in 
subsection \ref{limit-tree}.

\begin{rem}
\label{powers-trees}
It follows directly from Proposition \ref{prop:LL} that both, $T_+$ and $T_-$, are invariant under passing to iterates of $\phi$:
$$T_+(\phi^t) = T_+(\phi) \quad {\rm and} \quad T_-(\phi^t) = T_-(\phi) \quad {\rm for \, \, any \, \, integer} \quad t \geq 1 \, .$$
\end{rem}

\section{Train track technology}
\label{train-track-section}

This section will be a reference section, organized by subheaders as a kind of glossary. Almost everything in this section is known, or within an $\epsilon$-neighborhood of known facts. It is only for the convenience of the reader that we assemble in this section what is needed later from basic train track theory.

The expert reader is encouraged to completely skip this section and only go back to it if he or she needs an additional explanation in the course of reading the later sections.

\begin{notation}
In order to avoid confusion, in this section the automorphisms in question will be denoted by $\alpha$, as in the later sections we will have to consider both cases, $\alpha =\phi$ and $\alpha = \phi^{-1}$. Similarly, a train track map will be called $f: \tau \to \tau$ (to be specified later to $f_+: \tau_+ \to \tau_+$ or $f_-: \tau_- \to \tau_-$, the forward limit tree will be simply called $T$ (rather than $T_+$ or $T_-$), etc.
\end{notation}

\medskip
\subsection{Graphs, paths and graph maps}
\label{tt-maps}

${}^{}$
\smallskip

In this paper $\tau$ will always denote a graph, 
specified as follows:

\begin{convention}
\label{graphs}
By a {\em graph} $\tau$ we mean a non-empty topological space which is equipped with a cell structure, consisting of vertices and edges. Furthermore, $\tau$ satisfies the following conditions, unless explicitly otherwise stated:
\begin{itemize}
\item
$\tau$  is connected.
\item
$\tau$ is a finite graph, i.e. it consists of finitely many vertices and edges.

(Of course, the universal covering $\tilde \tau$ of a finite graph $\tau$ is in general not finite, but this counts as ``explicitly stated exception''.) 
\item
There are no vertices of valence 1 in $\tau$ (but we do admit vertices of valence 2).
\end{itemize}
\end{convention}
 
We systematically denote vertices of $\tau$ by $v, v'$ or $v_i$, and edges by $e, e', e_j$ etc. A point which can be a vertex or an interior point of an edge is usually denoted by $P$ or $Q$.  

An edge $e$ of $\tau$ is always oriented, and we denote by $\bar e$ the conversely oriented edge. 
In an {\em edge path} $\ldots e_k e_{k+1} e_{k+2} \ldots$ we always write the edges $e_i$ in the direction in which the path runs. We occasionally use ${\rm Edges}(\tau)$ to denote the {\em edge set} of $\tau$. We use the convention that its elements are, as before, {\em oriented} edges, so that the edge set ${\rm Edges}(\tau)$ contains for each pair $e, \bar e$ only the element $e$. 
If need be, we use the notation ${\rm \overline{Edges}}(\tau) := \{ \bar e \mid e \in {\rm Edges}(\tau) \}$.
Note that one is free, at any convenient time, to reorient edges: this operation does not change the graph, in our understanding, only the extra information about notation of edges. 

However: to be specific: if we say ``an edge of $\tau$'', this may well be the edge $\bar e$, for $e \in {\rm Edges}(\tau)$.

\smallskip

We will denote the {\em simplicial length} of an edge path $\gamma$, i.e. the number of edges traversed by $\gamma$, with $L(\gamma)$.  For the edge path $\bar \gamma$, by which we mean the path $\gamma$ with reversed orientation, one has of course $L(\bar \gamma) = L(\gamma)$.  In particular, for any single edge $e$ one has $L(e) = L(\bar e) = 1$.

\begin{convention-defn}
\label{paths}
(a) If we use the terminology {\em path} without further specification, we always mean a finite path.  The reader is free to formalize paths according to his or her own preferences; in most cases we recommend the viewpoint where a path $\gamma$ means an {\em edge path} $\gamma = e_1 \circ e_2 \circ \ldots \circ e_q$, where the $e_i$ are edges of $\tau$ such that the terminal vertex of any $e_i$ coincides with the initial vertex of $e_{i+1}$. According to circumstances, we generalize this concept sometimes slightly by allowing that $e_1$ (or $e_q$) is a non-degenerate terminal (or initial) segment of an edge (or, if $q = 1$, any segment of a single edge).

A path $\gamma$ is called {\em trivial} or {\em degenerate} if it consists of a single point.
Otherwise, $\gamma$ is said to be {\em non-trivial} or {\em non-degenerate}.

We will also sometimes suppress the distinction between an edge (or edge segment) $e$ and a path which traverses precisely $e$.

\smallskip
\noindent
(b)
A path $\gamma'$ is called a {\em backtracking path} if its endpoints coincide, and if the issuing loop is contractible in $\tau$. If $\gamma'$ occurs as subpath of a path $\gamma$,
then $\gamma'$ is called a {\em backtracking subpath} of $\gamma$.

\smallskip
\noindent
(c)
A (possibly infinite or biinfinite) path $\gamma$ is {\em reduced} if it doesn't contain any 
non-degenerate 
backtracking subpath. We denote by $[\gamma]$ the reduced path obtained from a possibly non-reduced path $\gamma$ by iteratively contracting all non-degenerate backtracking subpaths.
(It is well known that $[\gamma]$ depends only on $\gamma$ and not on the sequence of iteratively contracted backtracking subpaths,
if $\gamma$ is finite, or if $\gamma$ is the image of a reduced infinite or biinfinite path under a homotopy equivalence of $\tau$).

\smallskip
\noindent
(d)
Two subpaths $\gamma_1$ and $\gamma_2$ of a path $\gamma$ are {\em disjoint
on $\gamma$} if in the course of traversing $\gamma$ one first traverses $\gamma_1$ completely without starting $\gamma_2$, or conversely. Similarly we define the {\em overlap} of the two subpaths as the maximal subpath of $\gamma$ which is also a subpath of both $\gamma_i$.
\end{convention-defn}

Recall from the beginning of section \ref{new-background} that
a finite connected graph $\tau$ is called a {\em marked graph}, if it is equipped with a {\em marking isomorphism} $\theta: \FN \overset{\cong}{\longrightarrow} \pi_1(\tau)$.  Recall also that we purposefully suppress the issue of choosing a basepoint of $\tau$, as 
we are 
mainly
interested in automorphims of $\FN$ up to inner automorphisms.

\begin{convention-defn}
\label{maps-of-graphs}
(a) A {\em map} $f: \tau \to \tau'$ between graphs $\tau$ and $\tau'$
will always map vertices to vertices and edges to edge paths, which a priori may be non-reduced.

\smallskip
\noindent
(b)
A self-map $f: \tau \to \tau$ of a graph $\tau$, provided with a {marking isomorphism} $\theta: \FN \to \pi_1(\tau)$, {\em represents} an automorphism $\alpha$ of $\FN$ if the induced automorphisms $f_*: \pi_1 (\tau) \to \pi_1 (\tau)$ satisfies $\theta \circ \alpha = f_* \circ \theta$ up to inner automorphisms.  
\end{convention-defn}

\begin{rem}
\label{bounded-back}
It is well known \cite{Co} that for any map $f: \tau \to \tau'$  between graphs $\tau$ and $\tau'$, which induces an isomorphism on $\pi_1(\tau)$ (or, for the matter, a monomorphism), there is an upper bound to the length of any backtracking path $\gamma'$ which is contained as subpath in the (non-reduced) image $f(\gamma)$ of a reduced path $\gamma$ in $\tau$.
\end{rem}

\begin{defn}
\label{expanding-and-D}
(a)
A self-map $f: \tau \to \tau$ is called {\em expanding} if for every edge $e$ of $\tau$ there is an exponent $t \geq 1$ such that the edge path $f^t(e)$ has length $L(f^t(e)) \geq 2$.

\smallskip
\noindent
(b)
If $f: \tau  \to \tau$ is expanding, then there is a well defined self-map $Df$ on the set
${\rm Edges}(\tau) \cup \overline{\rm Edges}(\tau)$
which associates to every edge $e$ 
the initial edge of the edge path $f(e)$. 
\end{defn}

\begin{rem}
\label{make-expanding}
(1) If a self-map $f: \tau \to \tau$ represents an iwip automorphism of $\FN$, then the hypothesis that $f$ be expanding is always easy to satisfy:  It suffices to contract all edges which are not expanded by any iterate $f^t$ to an edge path of length $\geq 2$:  The issuing contracted subgraph must be a forest, as otherwise some $f^t$ will fix (up to conjugacy) a non-trivial proper free factor of $\pi_1(\tau)$.

\smallskip
\noindent
(2)
Alternatively, 
one can reparamatrize the map $f$ along each edge so that it becomes a true homothety with respect to the Perron-Frobenius metric 
defined as follows:  The {\em transition matrix} $M(f) = (m_{e, e'})_{e, e' \in Edges(\tau)}$ of $f$ is defined by setting $m_{e, e'}$ to be the number of times that $f(e')$ crosses over $e$ or over $\bar e$ (in both cases counted positively !). We will see in the subsequent lemma that this non-negative matrix is always primitive, so that there is up to positive multiples only one real row eigenvector of $M(f)$ with positive coefficients.  This {\em Perron-Frobenius} eigenvector $\vec v^{\,*} = (v_e)_{e \in Edges(\tau)}$ can be used to define a {\em Perron-Frobenius length} $L^{PF}(e) := v_e$ for all edges of $\tau$.
For more details see for example \cite{Lu2}, \S 3. 
\end{rem}

Since in this paper we will concentrate on self-maps of graphs that represent iwip automorphisms, we will almost exclusively consider self-maps $f: \tau \to \tau$ that are expanding. 

\begin{lem}
\label{invariant-subgraph}
Let the self-map $f: \tau \to \tau$ of the graph $\tau$ be expanding, and assume that $f$ represents an iwip automorphism of $\FN$. Then we have:
\begin{enumerate}
\item
Every $f$-invariant subgraph $\tau_0 \subset \tau$ is either a single vertex 
(or a collection of vertices, if one relaxes the convention that a graph is connected), or equal to all of $\tau$.
\item
The transition matrix $M(f)$ of $f$ is primitive.
\item
For every edge $e$ there exists an exponent $t \geq 0$ such that the (possibly non-reduced) edge path $f^t(e)$ crosses three or more times over $e$ or $\bar e$.
\end{enumerate}
\end{lem}

\begin{proof}
(1)
Assume that $\tau_0 \subset \tau$ contains at least one edge. From the assumptions that $f$ is expanding, and that $f(\tau_0) \subset \tau_0$, it follows that $\pi_1 (\tau_0)$ is non-trivial.
From the iwip hypothesis it follows that 
the inclusion $\tau_0 \subset \tau$ induces an isomorphisms on $\pi_1$. Hence our convention on graphs (see subsection \ref{tt-maps}), that $\tau$ is finite and does not contain vertices of valence 1, implies $\tau_0 = \tau$.

\smallskip
\noindent
(2)
If $M(f)$ is not primitive, then some positive power of it is reducible, which means precisely that there is an invariant proper subgraph. But we proved in (1) that this is impossible for $f$; since positive powers of expanding train track maps are again expanding train track maps, and powers of iwips are again iwips, the proof of (1) applies to all $f^t$ with $t \geq 1$.

\smallskip
\noindent
(3)
This is a direct consequence of the well known fact that for a primitive matrix the size of every coefficient tends to $\infty$ when taking large powers.
\end{proof}

\medskip

\subsection{Gates, turns, train tracks}
\label{gates-turns-tts}

${}^{}$
\smallskip

\begin{defn}
\label{gates}
(a)
For any expanding self-map of graphs $f: \tau \to \tau$ and any vertex $v$ of $\tau$ one partitions the edges with initial vertex $v$ into equivalence classes, called {\em gates} $\frak g_i$, by the following rule:  Two edges $e_1$ and $e_2$, both with initial vertex $v$, belong to a common gate $\frak g_i$ if and only if 
there is 
an exponent $t \geq 1$, such that $Df^t(e_1) = Df^t(e_2)$. To be specific, in case of a {\em loop edge} $e$ at $v$ (i.e. the initial and the terminal vertex of $e$ both coincide with $v$) the edges $e$ and $\bar e$ count as distinct edges with initial vertex $v$, which can or cannot belong to the same gate at $v$.

\smallskip
\noindent
(b)
This definition is sometimes extended to points $P$ in the interior of an edge $e$, at which there are precisely two gates, one containing precisely the terminal segment $e''$ of $e$ which starts at $P$, and the other one containing precisely the edge segment $\bar e'$, where $e'$ is the initial segment of $e$ which terminates at $P$.
\end{defn}

\begin{rem}
\label{gates1}
For any vertex (or point $P$) of $\tau$ the map $f$ induces via the map $Df$ 
a map 
$f^P_\frak G$ 
from the gates at $P$ to the gates at $f(P)$. It follows directly from Definition \ref{gates} (a) that this map $f^P_\frak G$ is injective, and hence, if $P$ is a periodic point, that $f^P_\frak G$ is bijective.
\end{rem}

\begin{defn}
\label{turns}
(a)
A pair of edges of $\tau$ forms a {\em turn} $(e, e')$ at a vertex $v$ if and only if both, $e$ and $e'$, have $v$ as initial vertex.  The turn is said to be {\em degenerated} if $e = e'$. Otherwise it is called {\em non-degenerate}.

\smallskip
\noindent
(b)
A path $\gamma$ {\em crosses over a turn} $(e, e')$ (or {\em contains the turn} $(e, e')$) if $\gamma$ contains the edge path $\bar e \circ e'$  as subpath (or a path $\bar e_0 \circ e_0'$ for non-degenerate initial segments $e_0$ of $e$ and $e'_0$ of $e'$).
\end{defn}

Note that a path $\gamma$ is reduced if and only of it doesn't cross over any degenerate turn.

\begin{defn}
\label{legal-paths}
Let $f: \tau \to \tau$ be an expanding self-map of a graph $\tau$.
\begin{enumerate}
\item
The map $f$ induces canonically
a map $D^2 f$ on turns, by setting $D^2((e, e')) = (Df(e), Df(e'))$.

\item
A turn $(e, e')$ is called {\em legal} if for all $t \geq 0$ the image turn $D^2f^t((e, e'))$ is non-degenerate. Otherwise $(e, e')$ is called {\em illegal}. In particular, all degenerate turns are illegal.  

A path (or loop) $\gamma$ in $\tau$ is called {\em legal} if it crosses only over legal turns. Otherwise it is called {\em illegal}.

\smallskip
\noindent
(From Definition \ref{gates} it follows directly that the turn $(e, e')$ is legal if and only if $e$ and $e'$ belong to distinct gates at their common initial vertex.)

\item
The map $f$ 
is called a {\em train track map} 
if for every edge $e$ of $\tau$ the edge path $f(e)$ is legal.  (It is easy to see that this condition is equivalent to the requirement that $f$ maps any legal path $\gamma$ to a legal path $f(\gamma)$.)
\end{enumerate}
\end{defn}

\begin{rem}
\label{several-legal}
(a)
It follows directly from the Definition \ref{legal-paths} (2) that $D^2f$ maps legal turns to legal turns, and illegal turns to illegal turns.

\smallskip
\noindent
(b)
If $f: \tau \to \tau$ is not expanding, then one can use the following definition of {\em legal paths}, which in the expanding case is equivalent to the one given in Definition \ref{legal-paths} (2):

{\em A reduced path $\gamma$ in $\tau$ is said to be {legal} if 
the maps $f^t$ are locally injective along $\gamma$, for all $t \geq 1$.}

\smallskip
\noindent
(c)
In some circumstances it can be useful to consider more general ``train track maps'', i.e. self-maps of graphs which have the {\em train track property} that legal maps are mapped to legal paths (where ``legal paths'' are defined as in part (b) above).  However, in this paper we insist on the assumption that a train track map $f: \tau \to \tau$ is always expanding ! 

\end{rem}

We now state a crucial property of train track maps, which follows directly from the Definition \ref{legal-paths}.
Recall from Definition-Convention \ref{paths} (c) that $[\gamma]$ denotes the path obtained from reducing a given path $\gamma$.

\begin{rem}
\label{ILT-defn}
For any train track map $f: \tau \to \tau$ and any path $\gamma$ in $\tau$ 
the number of illegal turns which are crossed over by $\gamma$, denoted by $\ILT(\gamma)$, satisfies:
$$\ILT([f(\gamma)]) \leq \ILT(f(\gamma)) = \ILT(\gamma)$$
Of course, a path $\gamma$ is legal if and only if $\ILT(\gamma) = 0$.
\end{rem}

\begin{lem}
\label{2-gates}
If $f: \tau \to \tau$ is a train track map which represents an iwip automorphism of $\FN$,
then at every vertex of $\tau$ there are at least 2 gates.
\end{lem}

\begin{proof}
Let $v$ be any vertex of $\tau$, and let $e$ be an edge with initial vertex $v$. Since $f$ is a train track map, the path $f^t(e)$ must be legal, for any $t \geq 0$. By Lemma \ref{invariant-subgraph} (3) for some value of $t$ the path $f^t(e)$ contains $e$ (or $\bar e$) as interior edge. Let $e'$ be the edge adjacent to this occurrence of $e$ (or $\bar e$) on $f^t(e)$, so that $e' e$ or $\bar e \,\bar e'$ is a subpath of $f^t(e)$.  By Definition \ref{legal-paths} (2) the edges $\bar e'$ and $e$ belong to distinct gates (at the vertex $v$).
\end{proof}

The following is one of the fundamental results for automorphisms of free groups (slightly adapted to the language specified in this section):

\begin{thm}[\cite{BH}]
\label{Bestvina-Handel}
For every iwip automorphism $\alpha$ of $\FN$ there exists a train track map $f: \tau \to \tau$ that represents $\alpha$.
\end{thm}

\medskip
\subsection{Eigenrays}
\label{rays}

${}^{}$
\smallskip

Throughout this subsection we assume that $f: \tau \to \tau$ is a 
train track map.

\begin{defn-rem}
\label{eigenrays} 
(a)
A {\em ray} $\rho$ is an infinite reduced path in $\tau$ which doesn't necessarily start at a vertex. The path $\rho$ can alternatively be thought of as locally injective map $\R_{\geq 0} \to \tau$ which crosses infinitely often over vertices, or as reduced infinite edge path $e'_1 e_2 e_3 \ldots$, where the $e_i$ with $i \geq 2$ are edges of $\tau$, and $e'_1$ is a non-degenerate terminal segment of some edge $e_1$ of $\tau$ (which includes the possibility $e'_1 = e_1$). The point $P \in \tau$ (not necessarily a vertex~!) which is the inital point of $e'_1$ (from the second viewpoint) or the image of $0 \in \R_{\geq 0}$ (from the first viewpoint) is called the {\em starting point} of $\rho$.  If $P$ is a vertex, then we also call it the {\em initial vertex} of $\rho$.

\smallskip
\noindent
(b)
A ray $\rho$ is called an {\em eigenray} if one has $f^t(\rho) = \rho$ 
for some integer $t \geq 1$. (Note that the condition $f^t(\rho) = \rho$ is stronger than requiring just $[f^t(\rho)] = \rho$~!)  In particular, it follows that the starting point $P$ of $\rho$ is $f$-periodic, and that $\rho$ is a legal path.
\end{defn-rem}

\begin{lem}
\label{eigen-path}
Let $\gamma$ be a non-trivial 
finite path in $\tau$, and assume that for some $t \geq 1$ the map $f^t$ maps $\gamma$ to a path $f^t(\gamma)$ which contains $\gamma$ as initial subpath. Then there is precisely one eigenray $\rho$ in $\tau$ which has $\gamma$ as initial subpath.
\end{lem}

\begin{proof}
It is easy to see that the union of the $f^{kt}(\gamma)$ for all $k \geq 1$ form an eigenray. Conversely, any eigenray $\rho$ which contains $\gamma$ as initial subpath must also contain any $f^{kt}(\gamma)$ as initial subpath.
\end{proof}

We see from Lemma \ref{eigen-path} that an eigenray can never bifurcate in the forward direction to give rise to two eigenrays with same initial point. However, we will see later that, in the 
backwards 
direction, an eigenray may well bifurcate, giving rise to an INP (compare Lemma \ref{two-INPs}).  

\smallskip

Recall (see Remark \ref{gates1}) that  
for each point $P$ of $\tau$ the map $f$ induces a map $f^P_\frak G$ on the gates at $P$.

\begin{prop}
\label{gates-eigenrays}
Let $f: \tau \to \tau$ be a train track map, and let $P$ be a periodic vertex (or interior point) of $\tau$. Then every gate $\frak g_i$ at $P$ is mapped by 
$f^P_\frak G$ 
periodically. A gate $\frak g_i$ contains precisely one 
edge $e_i$ 
(or edge segment $e'_i$)
on which $Df$ acts periodically (an 
``eigen edge''), and 
there is precisely one eigenray $\rho$ which starts ``from $\frak g_i$'', i.e. which starts with an edge (or edge segment) that belongs to $\frak g$. This edge is precisely the eigen edge $e_i$.
\end{prop}

\begin{proof}
By Definition \ref{gates} for every gate $\frak g$ there is an exponent $t \geq 0$ such that $D f^t$ maps every edge of $\frak g$ to a single edge. Recall from Remark \ref{gates1} that  
for each periodic point $P$ of $\tau$ 
the map $f^P_\frak G$ induced by $f$ on the gates at $P$ is bijective.
Hence each gate $\frak g_i$ at $P$ contains precisely one edge $e_i$ 
(the ``eigen edge'') 
which is periodic under the induced map $Df$, say with period $k(i) \in \N$. 

Thus $f^{k(i)}(e_i)$ is an edge path with initial subpath $e_i$. Hence Lemma \ref{eigen-path} gives us a well defined eigenray which starts from the gate $\frak g_i$, with initial edge $e_i$.

Since the initial edge (or edge segment) of an eigenray is necessarily periodic under the map $Df$, and two eigenrays with same initial edge (or edge segment) must agree (see Lemma \ref{eigen-path}), it follows that from every gate only one eigenray can start. 
\end{proof}

From Proposition \ref{gates-eigenrays} we obtain directly the following:

\begin{cor}
\label{eigen-bijection}
For every periodic point $P$ of $\tau$ there is a canonical bijection between (i) the gates at $P$, (ii) the eigen edges with initial vertex $P$, and (iii) the eigenrays with starting point $P$.
\qed
\end{cor}

We will also need later the following property:

\begin{lem}
\label{non-periodic}
Let $f: \tau \to \tau$ be a train track map which represents an automorphism of $\FN$, and let $\rho$ be an eigenray in $\tau$. Then $\rho$ can not be an eventually periodic path, i.e. a path of the form $\gamma_0 \circ \gamma_1 \circ \gamma_1 \circ \gamma_1 \circ \ldots$
(where $\gamma_0$ can possibly be trivial).
\end{lem}

\begin{proof}
By our convention from Definition \ref{legal-paths} (3) the map $f$ is expanding, which contradicts the fact that the loop $\gamma_1$ would have to be mapped to itself, given that by hypothesis $f$ represents an automorphisms of $\FN$.
\end{proof}

\medskip
\subsection{INPs}
\label{INPs}

${}^{}$
\smallskip

As before, assume that throughout this subsection $f: \tau \to \tau$ is a train track map. 
In addition, we assume in this section that $f$ induces an automorphism on $\pi_1(\tau)$.

\begin{defn}
\label{defn-INPs}
A path $\eta$ in $\tau$ which crosses over precisely one illegal turn is called a {\em periodic indivisible Nielsen path} (or {\em INP}, for short), if for some exponent $t \geq 1$ one has $[f^t(\eta)] = \eta$, (where, as before, $[\gamma]$ denotes the path obtained via reduction from a possibly unreduced path $\gamma$).

The illegal turn on $\eta = \gamma' \circ \bar \gamma$ is called the {\em tip} of $\eta$, while the two maximal initial legal subpaths $\gamma'$ and $\gamma$, of $\eta$ and $\bar \eta$ respectively, are 
called the {\em branches} of $\eta$. 
\end{defn}

If below we write an INP $\eta = \gamma' \circ \bar \gamma$, then unless otherwise specified, $\gamma'$ and $\gamma^{-1}$ denote the two legal branches of $\eta$.

\begin{rem}
\label{train-track-maps}
Let $\eta = \gamma' \circ \bar \gamma$ be an INP of the train track map $f: \tau \to \tau$.
\begin{enumerate}
\item
\label{two}
The two 
endpoints of $\eta$ are fixed points (but not necessarily vertices~!) of $f^t$ for some $t \geq 1$. The branches $\gamma$ and $\gamma'$ of $\eta$ are initial segments of eigenrays $\rho$ and $\rho'$ respectively, defined as unions of the nested sequences of paths $(f^{kt}(\gamma))_{k \in \N}$ and $(f^{kt}(\gamma'))_{k \in \N}$ (compare Lemma \ref{eigen-path}). The rays $\rho$ and $\rho'$ coincide up to the initial segments $\gamma$ and $\gamma'$ respectively: indeed, the common terminal segment of $\rho$ and $\rho'$ is precisely the union of all the paths crossed back and forth by the backtracking subpaths (see Convention-Definition \ref{paths} (b)) of the unreduced paths $f^{kt}(\eta)$, at the tip of the illegal turn of $\eta$, for all $k \geq 1$.
\item
\label{three}
Two eigenrays $\rho$ and $\rho'$, which coincide up to initial segments $\gamma$ and $\gamma'$ respectively, define always an INP $\eta = \gamma' \circ \bar \gamma$ as above in (\ref{two}). This follows directly from the definitions.
\end{enumerate}
\end{rem}

\begin{rem}
\label{finitely-many-INPs}
For every train track map $f: \tau \to \tau$ there are only finitely many INPs in $\tau$, if one assumes (as done throughout this section) that (i) $f$ is expanding, (ii) $\tau$ is finite, and (iii) $f$ 
induces an automorphism on $\pi_1(\tau)$.

This is a consequence of the fact that $f$ induces a quasi-isometry on the universal covering space (with respect to the simplicial metric $L$), so that for any reduced path in $\tau$ the length of any backtracking subpath in the image path 
is bounded above by a constant depending only on $f$. Since $f$ is expanding, this gives a bound to the maximal length of the legal branches $\gamma$ and $\gamma'$ of any  INP $\eta = \gamma' \circ \bar \gamma$ in $\tau$, since the only backtracking subpath in $f(\eta)$ is the subpath cancelled at the tip when $f(\eta)$ is reduced to $[f(\eta)]$. 

(If the reader wants to fill in the details, we recommend restricting to the case where $f$ induces an iwip automorphism on $\pi_1(\tau)$ and working with the PF-metric $L^{PF}$ introduced in subsection \ref{limit-tree} rather than with the simplicial metric $L$; by Remark \ref{quasi-isometry} the two metrics define a quasi-isometry for $\tilde \tau$.)
\end{rem}

\begin{convention}
\label{subdivision-INP}
As pointed out in Remark \ref{train-track-maps} (\ref{two}) the endpoints of an INP may a priori well lie in the interior of an edge. However, since these points are $f$-periodic, and since by Remark \ref{finitely-many-INPs} there are only finitely many INPs in $\tau$, we can assume from now on that the edges of $\tau$ have been subdivided accordingly, so that all endpoints of INPs are vertices.

(Of course, such a subdivision must be followed potentially by the procedure lined out in Remark \ref{make-expanding}, in order to make sure that after subdivision the train track map $f$ is still expanding.)
\end{convention}

\begin{lem}
\label{two-INPs}
There exists an exponent $r_1 \geq 1$ with the following property:
Let $\gamma$ be a path in $\tau$, and assume that it contains precisely two illegal turns, each being the tip of an INP-subpath $\eta_1$ and $\eta_2$ of $\gamma$.  Assume that $\eta_1$ and $\eta_2$ overlap in a non-degenerate  subpath.  Then $f^{r_1}(\gamma)$ reduces to a path $[f^{r_1}(\gamma)]$ which is legal.
\end{lem}

\begin{proof}
By hypothesis, $\eta_{1} = \gamma'_1 \circ \bar \gamma_1$ and $\eta_{2} = \gamma'_2 \circ \bar \gamma_2$ intersect in a non-degenerate path, which by Convention \ref{subdivision-INP} must be an edge path $\gamma'$ with vertices as initial and terminal point. 

Each INP contains only one illegal turn, 
and by assumption $\gamma$ contains two
such. Since $\gamma'$ is a subpath of both, $\eta_1$ and $\eta_2$, it follows that 
the path $\gamma'$ cannot cross over either of the two illegal turns on $\gamma$.
Hence $\gamma'$
must be a legal subpath of both, the branch $\gamma'_2$ of $\eta_2$, and the inverse of the branch $\gamma_1$ of $\eta_1$.

Since by assumption $\gamma'$ is non-degenerate, it is expanded under iteration of $f$ to a legal edge path of arbitrary big length.

Assume that an exponent $r \geq 1$ is chosen such that $L(f^r(\gamma'))$ is strictly bigger than 
the sum $L(\gamma_1) + L(\gamma'_2)$.
In this case 
we observe that on $f(\gamma')$ there must be a non-degenerate subpath which belongs to both, the backtracking subpath on $f^r(\gamma_1^{-1})$ at the tip of $f^r(\eta_1)$, and the backtracking subpath on $f^r(\gamma'_2)$ at the tip of $f^r(\eta_2)$.
It follows then directly 
that, in reducing $f^r(\gamma)$, both illegal turns disappear into the backtracking subpaths, so that the resulting path $[f^r(\gamma)]$ is legal.

Since $f$ is expanding and since (see Remark \ref{finitely-many-INPs}) there are only finitely many INPs in $\tau$, it is easy to find a bound $r_1$ as required in the claim.
\end{proof}

The following is one of the crucial properties of train track maps of graphs. It goes back to the first paper \cite{BH} on the subject. An alternative proof is given in \cite{Lu1}.

\begin{prop}
\label{quote}
For every train track map $f: \tau \to \tau$
there exists a constant $r_2 = r_2(\tau) \geq 0$ such that every path $\gamma$ with precisely 1 illegal turn satisfies:

\smallskip
{Either $\gamma$ contains an INP as subpath, or else $[f^{r_2}(\gamma)]$ is legal.}
\qed
\end{prop}

Recall from Remark \ref{ILT-defn} that $\ILT(\gamma)$ denotes the number of illegal turns in a path $\gamma$.

\begin{prop}
\label{ILT-decrease}
There exists an exponent 
$r = r(f) \geq 0$ such that every finite path $\gamma$ in $\tau$ 
with $\ILT(\gamma) \geq 1$ 
satisfies
$$\ILT([f^r(\gamma)]) < \ILT(\gamma) \, ,$$
unless every illegal turn on $\gamma$ is the tip of an INP-subpath $\eta_i$ of $\gamma$, where any two $\eta_i$ are either disjoint subpaths on $\gamma$, or they overlap precisely in a common endpoint (for this terminology compare Convention-Definition \ref{paths} (d)).
\end{prop}

\begin{proof}
It suffices to take $r$ to be the maximum of $r_1$ from Lemma \ref{two-INPs} and $r_2$ from Proposition \ref{quote}. We then apply the latter to any maximal subpath with precisely one illegal turn. If each such subpath contains an INP-subpath, we apply Lemma \ref{two-INPs}
to any maximal subpath which contains precisely two illegal turns for which the corresponding INP-subpaths overlap in at least on edge.
\end{proof}

Bestvina-Handel \cite{BH} proved that every iwip automorphism $\phi$ can be represented by a train track map $f: \tau \to \tau$ such that (up to inversion) there is at most one indivisible Nielsen path in $\tau$ which is fixed by $f$. However, since in this paper we are interested in periodic indivisible Nielsen paths (INP's) rather than fixed ones, these {\em stable} train track maps are not good enough for our purposes.

Nevertheless, the number of INP's in any train track representative of $\phi$ is bounded above by a constant which only depends on $N$, since every INP contributes non-trivially to the {\em index} of $\phi$, which is bounded above by 2N-2, see \cite{GJLL}.  Thus one can apply the above quoted result of \cite{BH} to a suitable power of $\phi$ to obtain:

\begin{prop}
\label{at-most-one-INP}
For every iwip automorphism $\phi$ there is a positive power $\phi^k$ which is represented by a train track map $f: \tau \to \tau$ such that in $\tau$ there is (up to inversion) at most one INP $\eta$.
The endpoints of $\eta$ coincide 
to give a closed loop if and only if $\phi$ is 
induced by a homeomorphism of some surface with precisely one boundary component or puncture. 
\qed
\end{prop}

Recall that any iwip $\phi$ which is 
induced by a homeomorphism of some surface with precisely one boundary component or puncture can not be atoroidal, see Definition \ref{defn:iwips}, since the boundary curve gives rise to a non-trivial $\phi$-periodic conjugacy class.

\medskip
\subsection{Used turns}
\label{used-turns}

${}^{}$
\smallskip

As before, let $f: \tau \to \tau$ be a train track map which is fixed throughout this subsection.

\begin{defn}
\label{six}
A 
turn $(e, e')$ in $\tau$ is called 
{\em used}\,\footnote{\, Such turns have been termed ``taken'' in C. Pfaff's work which in the mean time has become publicly available, see 
\cite{Pfaff}.}
if there is an edge $e''$ in $\tau$ with the property that for some $t \geq 1$ the image path $f^t(e'')$ {crosses over} this turn.
Otherwise the 
turn $(e, e')$ is called {\em unused}. 
\end{defn}

\begin{rem}
\label{rem-used}
From Definition \ref{six} we derive directly the following facts:
\begin{enumerate}
\item
Every used turn is legal. The converse is in general wrong.
\item
The image turn (under the map $D^2f$) of any used turn is also used. 
\item
The image of an unused turn $(e, e')$ may well be used. There is, however, a constant $s \geq 0$ (which only depends on the total number of turns in $\tau$) such that either $D^2f^s(e, e')$ is used, or else all forward iterates of $(e, e')$ under $D^2f$ are unused.
\end{enumerate}
\end{rem}

Since we know from subsection \ref{rays} that every eigenray $\rho$ in $\tau$ is for some integer $k \geq 1$ the infinite union of all $f^{kt}(e)$, where $t \geq 1$ and $e$ is the initial edge (or edge segment) of $\rho$, we obtain directly:

\begin{lem}
\label{eigenrays-used}
Every eigenray $\rho$ of a train track map $f$ crosses only over used turns.
\qed
\end{lem}

\begin{lem}
\label{prolongation}
If the train track map $f: \tau \to \tau$ represents an iwip automorphism, then for every edge $e$ of $\tau$ 
there is an edge $e'$ which has the same initial vertex as $e$, 
such that $\bar e e'$ is a used legal turn. 
\end{lem}

\begin{proof}
This is a direct consequence of Lemma \ref{invariant-subgraph} (3).
\end{proof}

\begin{lem}
\label{used-gate-turns}
Let $\frak g_1$ and $\frak g_2$ be two gates at the same periodic vertex $v$ of $\tau$, and assume that some turn $(e'_1, e'_2)$ with $e'_i \in \frak g_i$ is used. Then the turn $(e_1, e_2)$ is also used, where $e_i$ is the initial edge of the eigenray $\rho_i$ that starts from the gate $\frak g_i$, for $i = 1$ and $i = 2$. (By Proposition \ref{gates-eigenrays} this is equivalent to stating that each edge $e_i$ is the eigen edge of the gate $\frak g_i$.)
\end{lem}

\begin{proof}
From the definition of a gate it follows directly that for every periodic gate $\frak g_i$ there is an exponent $t_i \geq 1$ such that $D f^{t_i}$ maps every edge $e'$ in $\frak g_i$ to the eigen edge $e_i$ of $\frak g_i$.

Now, if the turn $(e'_1, e'_2)$ is used, then for some edge $e$ and some $t \geq 1$ the edge path $f^t(e)$ crosses over the turn $(e'_1, e'_2)$. It follows that 
the path $f^{t t_1 t_2}(e)$ crosses over the turn $(e_1, e_2)$.
\end{proof}

\medskip
\subsection{BFH's ``stable'' lamination}
\label{section-BFH-lamination}

${}^{}$
\smallskip

The initials BFH in the following definition refer to Bestvina-Feighn-Handel, who introduced and studied the following lamination in \cite{BFH}; in particular they showed a 
special 
attraction property of this lamination, for train track maps which represent iwip automorphisms.
We give here a definition in more modern terms which is slightly more general and fits better in the set-up of this paper.

\begin{defn}
\label{BFH-lamination}
For any train track map $f: \tau \to \tau$ we define the {\em BFH-attracting lamination} $L^2_{BFH}(f)$ as the lamination which is generated 
(in the sense defined in subsection \ref{laminations}) by 
the family of paths $f^t(e)$, for any edge $e$ of $\tau$ and any exponent $t \geq 0$.
\end{defn}

It is easily seen that, if $f$ represents an automorphism $\alpha$ of $\FN$, that $L^2_{BFH}(f)$ is invariant under the natural map induced by $\alpha$ on $\partial^2 \FN$.

The following is a direct consequence of the definition of ``used turn'' in the previous section.

\begin{lem}
\label{BFH-used}
Let $f: \tau \to \tau$ be a train track map which represents an iwip automorphism. Then any leaf $(X, Y)$ of the lamination $L^2_{BFH}(f)$  is realized by a biinfinite path 
$\gamma = \gamma_\tau(X, Y)$
in $\tau$ which only crosses over used turns.
\qed
\end{lem}

The converse of this statement is not true: It is easy to find biinfinite reduced paths $\gamma$ in $\tau$ which only cross over used turns and which represent pairs $(X, Y) \in \partial^2\FN$, i.e. $\gamma = \gamma_\tau(X,Y)$, which do not belong to $L^2_{BFH}(f)$.  However, one can show (compare also Corollary \ref{no-singularity} (b)) that $(X, Y)$ does belong to $L^2_{BFH}(f)$ if for the whole $\phi$-orbit of $(X,Y)$ the geodesic realizations $\gamma_t = \gamma_\tau(\phi^t(X), \phi^t(Y))$ (with $t \in \Z$, see Remark \ref{no-lift}) only cross over used turns.

As a special case of this observation we have the following; an elementary proof derives as a direct consequence from our definitions and 
the facts stated in the paragraph just before Lemma \ref{eigenrays-used}:

\begin{lem}
\label{eigenray-BFH}
Any eigenray $\rho$ in $\tau$ generates a lamination $L^2(\rho)$ which satisfies
$L^2(\rho) \subset L^2_{BFH}(f)$.
\qed
\end{lem}

The following proposition is due to Bestvina-Feighn-Handel. 
We'd like to point out here 
that an alternative proof of part (1) follows as corollary from our 
proof of 
Theorem \ref{main-result2} 
in section \ref{steps-five-to-seven},
see Remark \ref{last}.

We would also 
like to recommend to the 
less experienced 
reader to try proving parts (2) and (3) as exercises, as they are quite doable 
(using the fact that the transition matrix $M(f)$ 
is primitive, see Lemma \ref{invariant-subgraph} (2)),
and as this will enhance the reader's understanding of the main subjects treated in this paper.

\begin{prop}[\cite{BFH}]
\label{from-BFH}
Let $f: \tau \to \tau$ be a train track representative of an iwip automorphism $\alpha$ of $\FN$.

\smallskip
\noindent
(1)
The lamination $L^2_{BFH}(f)$ depends only on $\alpha \in \Out(\FN)$ and not on the particular choice of the train track representative $f$.

\smallskip
\noindent
(2)
The lamination $L^2_{BFH}(f)$ is minimal (see Definition \ref{minimal-laminations}).

\smallskip
\noindent
(3)
One has $L^2_{BFH}(f^t) = L^2_{BFH}(f)$ for any integer $t \geq 1$.
\qed
\end{prop}

\medskip
\subsection{The limit tree and BBT}
\label{limit-tree}

${}^{}$
\smallskip

In this subsection we freely use the terminology and notation about $\R$-trees introduced previously in section \ref{new-background}.

\smallskip

Let $T$ be the forward limit tree of the iwip automorphism $\alpha$ (i.e.  $T \alpha = \lambda T$ for $\lambda := \lambda_+(\alpha) > 1$).  Let $f: \tau \to \tau$ be any train track representative of $\alpha $. 

It is well known (see \cite{GJLL}) that for the universal covering $\tilde \tau$ of $\tau$ there exists an $\FN$-equivariant surjective map $i: \tilde \tau \to T$ which is injective on legal paths (where a turn in $\tilde \tau$ is legal if and only if its image turn in $\tau$ is legal). If one uses the map $\tilde \tau$ to pull back the metric on $T$ to 
define an edge length 
for every edge $e$ of $\tilde \tau$ an thus (by the $\FN$-equivariance of $i$) also for the edges of $\tau$, then one obtains precisely the Perron-Frobenius length $L^{PF}(e)$ for some 
row
eigenvector $\vec v^{\, *} = (L^{PF}(e))_{e \in {\rm Edges}(\tau)}$ (with eigenvalue $\lambda$) of the non-negative primitive transition matrix $M(f)$, see 
Remark \ref{make-expanding} (2).
One obtains:

\begin{rem}
\label{legal-isometric}
(1)
The map $i: \tilde \tau \to T$ has the property that, with respect to the PF-metric $L^{PF}$, every finite legal path $\gamma$ in $\tilde \tau$ is mapped isometrically to the geodesic segment $i(\gamma)$ in $T$.

\smallskip
\noindent
(2)
In particular, every eigenray $\rho$ in $\tau$ has the property that any of its lifts $\tilde \rho$ in $\tilde \tau$ is mapped by the map $i$ isometrically to a ray in $T$ (which is an eigenray of a homothety $H: T \to T$, with stretching factor $\lambda$ as above, which realizes a suitable representative $\Phi \in \Aut(\FN)$ of $\phi$, see \cite{Lu2}, \S 4).

\smallskip
\noindent
(3)
Any INP $\eta = \gamma' \circ \gamma$ in $\tau$, on the other hand, lifts to paths $\tilde \eta = \tilde \gamma' \circ \tilde \gamma$ in $\tilde \tau$ which are completely ``folded'' by $i$:  one obtains $i(\tilde \gamma) = i(\tilde \gamma')$ in $T$.
\end{rem}

The map $i$ satisfies the {\em bounded backtracking property (BBT)}: There is a constant $\BBT(i) \geq 0$ such that for any two points $x, y \in \tilde \tau$ the geodesic segment $[x, y] \subset \tilde \tau$ is mapped by $i$ into the $\BBT(i)$-neighborhood of the geodesic segment $[i(x), i(y)]$:
$$i([x,y]) \subset \cal N_{BBT(i)}([i(x), i(y)])$$
Based on results of \cite{BFH} it has been shown in \cite{GJLL} that:
$$\BBT(i) \leq \vol^{PF}(\tau) := \sum_{e \in {\rm Edges}(\tau)} L^{PF}(e)$$

\begin{lem}
\label{long-legal-in-short}
Consider any element $w\in \FN$, with translation length $0 \leq || w ||_T < c$ on $T$, 
where $c$ is the smallest PF-length of any loop in $\tau$, and let $\hat \gamma$ be a reduced loop  in $\tau$ which represents the conjugacy class $[w] \subset \FN$. Let $\gamma'$ be a legal subpath of $\hat \gamma$. The PF-length of $\gamma'$ satisfies: 
$$L^{PF}(\gamma') \leq || w ||_T + 4 \BBT(i)$$ 
(Indeed, one can also show the stronger inequality $L^{PF}(\gamma') \leq || w ||_T$ + 2 \BBT(i))
\end{lem}

\begin{proof}
We can lift $\hat \gamma$ to a biinfinite geodesic $\tilde \gamma$ in the universal covering $\tilde \tau$, which is mapped by $i: \tilde \tau \to T$ to a (possibly non-reduced) path $i(\tilde \gamma) $ in $T$ that covers the axis ${\rm Ax}(w)$ in $T$. 
The action of $w$ on $\tilde \tau$ and on $T$ fixes $\tilde \gamma$ and ${\rm Ax}(w)$ and translates each of them (by the amount of $|| w ||_T$, for ${\rm Ax}(w)$).

From the assumption $||w||_T < c$ it follows 
(by Remark \ref{legal-isometric})
that $\hat \gamma$ is not a legal loop, so that in particular the legal subpath $\gamma'$ can not wrap around all of $\hat \gamma$. Let $P$ be a point on $\hat \gamma$ which is not contained in $\gamma'$, and let $\tilde P$ be a lift of $P$ to $\tilde \gamma$.
Hence the point 
$P_T := i(\tilde P)$ lies in the $\BBT(i)$-neighborhood of ${\rm Ax}(w)$. In particular it follows that $d(P_T, wP_T) \leq ||w||_T + 2\BBT(i)$.

We can now lift $\gamma'$ to a legal path $\tilde \gamma'$ which is contained in the geodesic segment $[\tilde P, w\tilde P] \subset \tilde \gamma$. Hence the endpoints of $i(\tilde \gamma')$ must lie in the $\BBT(i)$-neighborhood of $[P_T, w P_T]$. Thus they have distance at most $||w||_T + 4\BBT(i)$. But $\tilde \gamma'$ is legal and hence (Remark \ref{legal-isometric}) mapped isometrically to its $i$-image.  Thus $\tilde \gamma'$ and hence also $\gamma'$ have PF-length bounded above by $||w||_T + 4\BBT(i)$.
\end{proof} 

\begin{rem}
\label{quasi-isometry}
(1)
Note that with respect to the Perron-Frobenius length $L^{PF}$ any legal path $\gamma$ satisfies $L^{PF}(f(\gamma)) = \lambda L^{PF}(\gamma)$.  In particular, it follows that no edge $e$ can have $L^{PF}(e)$ equal to $0$,
since otherwise the subgraph $\tau_0 \subset \tau$ defined by all such edges would contradict Lemma \ref{invariant-subgraph} (1).

\smallskip
\noindent
(2)
As a consequence we observe that the simplicial edge length $L$ and the Perron-Frobenius edge length $L^{PF}$ give rise to quasi-isometric metrics on the universal covering $\tilde \tau$.
\end{rem}

\section{Steps 1 and 2}
\label{steps-one-and-two}

The material in this section is reminiscent to some of the techniques used previously by the second author in \cite{Lu3}.

Let $T_+$ be the forward limit tree of the atoroidal iwip automorphism $\phi$ (i.e.  $T_+ \phi = \lambda_+ T$ for $\lambda_+ > 1$).  Let $f_+: \tau_+ \to \tau_+$ be a train track representative of $\phi$. 
By Convention \ref{subdivision-INP} and Proposition \ref{at-most-one-INP}
we can assume 
the following,
after first having possibly replaced $\phi$ by a positive power of itself:

\begin{hyp}
\label{INP-hyp}
There is 
(up to inversion) 
at most one INP $\eta$ in $\tau_+$, and the two endpoints of $\eta$ are distinct vertices of $\tau_+$.
\end{hyp}

\begin{defn}
\label{totally-illegal}
For any constant $C \geq 1$ a (possibly infinite or bi-infinite) reduced path $\gamma$ in $\tau_+$ is called  {\em totally $C$-illegal} if every legal subpath $\gamma'$ of $\gamma$ has 
simplicial length (see subsection \ref{tt-maps})
$$L(\gamma') \leq C \, .$$
\end{defn}

\begin{rem}
\label{inequalities}
The Definition \ref{totally-illegal} is motivated by Lemma \ref{legal-bounded} below. Another, rather useful property of any finite totally $C$-illegal path $\gamma$, which follows directly from the definition, is given by the second of the following inequalities (while the first one doesn't require any hypotheses):
$$\ILT(\gamma) +1 \leq L(\gamma) \leq C(\ILT(\gamma) + 1)$$
\end{rem}

\begin{lem}
\label{legal-bounded}
There exists a constant $C \geq 1$ such that for every pair $(X,Y)$ in the dual lamination $L^2(T_+) \subset \partial^2\FN$ the reduced biinfinite path $\gamma = \gamma_{\tau_+}(X, Y)$ in $\tau_+$ (the ``geodesic realization'' of the pair $(X,Y)$, see Definition \ref{covering-versus-base})
is totally $C$-illegal.
\end{lem}

\begin{proof} 
By 
Remark \ref{dual-lamination-criterion}
a finite reduced path $\gamma'$ in $\tau_+$ is a subpath of the geodesic realization $\gamma_{\tau_+}(X,Y)$ for some $(X, Y) \in L^2(T_+)$ if and only if for every $\epsilon > 0$ there is an element $w \in \FN$ with translation length on $T_+$ of seize $||w||_{T_+} \leq \epsilon$, such that the conjugacy class of $w$ in $\FN$ is represented by a reduced loop $\hat \gamma$ in $\tau_+$ which contains $\gamma'$ as subpath. Hence the desired inequality is a direct consequence of Lemma \ref{long-legal-in-short}, where the constant $C$ can be calculated from the value $4 \BBT(i)$ for the map $\tilde \tau_+ \to T_+$ explained in subsection \ref{limit-tree} and the quasi-isometry constants between the length functions $L$ and $L^{PF}$, see Remark \ref{quasi-isometry} (2).
\end{proof}

Below we use the following terminology:  If $\gamma$ is a (possibly infinite or biinfinite) path in $\tau_+$, and $\gamma_1$ a subpath of $\gamma$, then a boundary subpath $\gamma'$ of $[f^t_+(\gamma_1)]$ (for some $t \geq 1$) is
{\em cancelled by the reduction of $f^t_+(\gamma)$ to $[f^t_+(\gamma)]$} if, for $\gamma = \gamma_0 \circ \gamma_1 \circ \gamma_2$, when $[f^t_+(\gamma_0)] \circ [f^t_+(\gamma_1)] \circ [f^t_+(\gamma_2)]$ is reduced to give $[f^t_+(\gamma)]$, then $\gamma'$ is entirely contained in one of backtracking subpaths at the concatenation points of the product path $[f^t_+(\gamma_0)] \circ [f^t_+(\gamma_1)] \circ [f^t_+(\gamma_2)]$.

\begin{lem}
\label{contraction}
For any constant $C \geq 1$ there exists an integer $s \geq 0$ with the following property:  
Let $(X,Y) \in \partial^2 \FN$ be such that for every integer $t \geq 0$ the geodesic realization 
$\gamma_t = \gamma_{\tau_+}(\phi^t(X),\phi^t(Y))$ (see Remark \ref{no-lift})
is totally $C$-illegal. 
Then 
for any 
any finite subpath $\gamma'$ of $\gamma_0$ with $\ILT(\gamma') \geq 2$ one has either
$$\ILT([f_+^s(\gamma')]) < \ILT(\gamma') \, ,$$
or else a boundary subpath of $[f_+^s(\gamma')]$ which crosses over at least one illegal turn is completely cancelled by the reduction of $f_+^s(\gamma)$ to $[f_+^s(\gamma)] = \gamma_s$.
\end{lem}

\begin{proof}
Recall first, from subsection \ref{laminations}, that $[f_+^{t}(\gamma_0)] = \gamma_t(X,Y)$.
From Proposition \ref{ILT-decrease} we know that for any finite subpath $\gamma'$ of $\gamma_0$ with 2 or more illegal turns
one has $\ILT([f_+^r(\gamma')]) < \ILT(\gamma')$ (for $r$ as given in Proposition \ref{ILT-decrease}), unless every illegal turn of $\gamma'$ is the tip of an INP entirely contained as subpath in $\gamma'$, such that any two such INP-subpaths which are adjacent on $\gamma'$ either (i) are disjoint, or (ii) intersect precisely in a common endpoint.

The second case (ii) is ruled out by our Hypothesis \ref{INP-hyp}, since $\gamma'$ is a reduced path and the only INP $\eta$ contained in $\tau_+$ is not closed.

In order to rule out the first case (i), we observe that the subpath $\gamma''$ of $\gamma'$, which connects the two endpoints of two adjacent INP-subpaths of $\gamma'$, has to run over at least one edge, and that $\gamma''$ is legal.  Hence, by the  expansiveness of $f_+$, for some $t \geq 1$ the path $[f_+^{t}(\gamma')]$ contains a legal subpath $f_+^t(\gamma'')$ of length bigger than the constant $C$.
This contradicts the hypothesis that $[f_+^t(\gamma)] = \gamma_t$ is totally $C$-illegal, unless one of the two INP-subpaths of $[f_+^t(\gamma')]$, which are on $[f_+^t(\gamma')]$ adjacent to the subpath $f_+^t(\gamma'')$, is completely contained in a boundary subpath of $[f_+^t(\gamma')]$ that is cancelled by the reduction of $f_+^t(\gamma)$ to $[f_+^t(\gamma)] = \gamma_t$.

Hence it suffices to take $s \geq r$ big enough so that $L(f^s(e)) > C$ for every edge $e$ of $\tau_+$.
\end{proof}

\begin{cor}
\label{strong-contraction}
Let $C \geq 1, (X,Y) \in \partial^2 \FN, \gamma'$ and $s \geq 1$ be as in Lemma \ref{contraction}.  Let us furthermore assume that $\ILT(\gamma') \geq 5$.  Then 
one has
$$\ILT([f_+^{s}(\gamma')]) \leq \frac{1}{4}\ILT(\gamma')$$
\end{cor}

\begin{proof}
By the assumption $\ILT(\gamma') \geq 5$ we can subdivide $\gamma' = \gamma_1 \circ \ldots \circ \gamma_q$ as concatenation of subpaths $\gamma_i$, where each $\gamma_i$ satisfies $2 \leq \ILT(\gamma_i)  \leq 3$, and where at each concatenation vertex the path $\gamma'$ crosses over an illegal turn. We now apply Lemma \ref{contraction} to each of the subpaths $\gamma_i$ and obtain, for $s \geq 1$ as specified there, that either $\ILT([f_+^s(\gamma_i)] < \ILT(\gamma_i)$, or else one of the illegal turns in $[f_+^s(\gamma_i)]$ is cancelled when $f_+^s(\gamma')$ is reduced to $[f_+^s(\gamma')]$. Thus one obtains, for
$$[f_+^s(\gamma')] = \gamma'_1 \circ \ldots \circ \gamma'_q \, ,$$
where each $\gamma'_i $ is the (possibly trivial) subpath of $[f_+^s(\gamma_i)]$ which is left-over when the concatenation $[f_+^s(\gamma_1)] \circ \ldots \circ [f_+^s(\gamma_q]$ is reduced,
that each of the paths $\gamma'_i$ satisfies $\ILT(\gamma'_i)  \leq \ILT(\gamma_i) - 1$. This gives $\ILT([f_+^s(\gamma')] < \frac{1}{4} \ILT(\gamma')$, as claimed.
\end{proof}

\begin{prop} 
\label{length-contraction}
For any $\lambda > 1$ there is a constant $c > 0$ and an exponent $t \geq 0$ with the following property:

Consider $(X,Y) \in L^2(T_+) \subset \partial^2 \FN$, and let $\gamma = \gamma_{\tau_+}(X,Y)$ be the geodesic realization as reduced biinfinite path in $\tau_+$. Let $\gamma'$ be a finite subpath of $\gamma$ of simplicial length $L(\gamma') \geq c$.
Then the reduced image path $[f_+^t(\gamma')]$ satisfies:
$$L([f_+^t(\gamma')]) \leq \frac{1}{\lambda} L(\gamma')$$
\end{prop}

\begin{proof}
By Lemma \ref{legal-bounded} there exists a constant $C \geq 0$ such that for any $t \in \Z$ the geodesic realization 
$\gamma_t := \gamma_{\tau_+}(\phi^t(X),\phi^t(Y))$ is totally $C$-illegal. Hence, for $s \geq 0$ as given by Lemma \ref{contraction}, any finite subpath $\gamma'$ of $\gamma_0$ with sufficiently large $\ILT(\gamma')$
satisfies by Corollary \ref{strong-contraction} the inequality 
$$\ILT([f_+^{ks}(\gamma')]) \leq \frac{1}{4^k}\ILT(\gamma')$$
for some large $k \geq 1$. 

Now, one can decompose $[f_+^{ks}(\gamma')]$ as concatenation 
$[f_+^{ks}(\gamma')] = \delta_0 \circ [f_+^{ks}(\gamma')]_{\gamma_{ks}} \circ  \delta_1$, where $[f_+^{ks}(\gamma')]_{\gamma_{ks}}$ is the maximal subpath of $[f_+^{ks}(\gamma')]$ which is also a subpath of 
the biinfinite reduced path ${\gamma_{ks}}$, while $\delta_0$ and $\delta_1$ are boundary subpaths that are cancelled when $f_+^{ks}(\gamma_0)$ is reduced to $[f_+^{ks}(\gamma_0)] = \gamma_{ks}$. In particular, it follows that 
for each of the two $\delta_i$ there is a backtracking subpath $\delta'_i$ of $f_+^{ks}(\gamma_0)$, such that $\delta_i$ is obtained from a subpath of $\delta'_i$ by canceling (in the process of reducing the subpath $f_+^{ks}(\gamma')$ of $f_+^{ks}(\gamma_0)$ to $[f_+^{ks}(\gamma')]$) certain further backtracking subpaths. Hence (compare Remark \ref{bounded-back}) $L(\delta'_i)$ and thus  $L(\delta_i)$ is bounded above by a constant $K \geq 0$ which depends only on $f_+^{ks}$.  

As subpath of the totally $C$-illegal path $\gamma_{ks}$ the path $[f_+^{ks}(\gamma')]_{\gamma_{ks}}$ must be itself totally $C$-illegal.
Hence we deduce, using Remark \ref{inequalities}, that
$$\begin{array}[t]{rcl}
L([f_+^{ks}(\gamma')])
&\leq&L([f_+^{ks}(\gamma')]_{\gamma_{ks}}) + 2K \\
&\leq&C(\ILT([f_+^{ks}(\gamma')]_{\gamma_{ks}}) + 1) +2K \\
&\leq&C(\ILT([f_+^{ks}(\gamma')]) + 1) +2K\\
&\leq&C(\frac{1}{4^k}\ILT(\gamma')) + 1) + 2K \\
&\leq&C(\frac{1}{4^k}L(\gamma')) + 1) +2K
\end{array}$$
It is now easy 
to determine first $k$ and thus $t$ large enough (in dependence of the given value of $\lambda$) and then (using again the inequality $L(\gamma) \leq C(\ILT(\gamma) + 1)$ from Lemma \ref{inequalities}) a sufficiently large constant $c$ which give the desired inequality.
\end{proof}

\begin{rem}
\label{contracting-stable}
It follows easily from standard arguments of geometric group theory that the strong uniform contraction property of the automorphism $\phi$ along the leaves of the dual lamination $L^2(T_+)$ of its forward limit tree $T_+$, which is stated in Proposition \ref{length-contraction} above in terms of the train track representative $f_+$ of $\phi$, is in fact an intrinsic property, which can be similarly expressed with respect to any representative of $\phi$ as a self-map of some marked graph, or, more algebraically, directly for the automorphisms $\phi$ expressed as traditionally by means of a basis $\cal A$ of $\FN$ and the image words of the generators $a_i \in \cal A$. Any $\phi$-invariant subset of $\partial^2\FN$ which satisfies this property is called {\em uniformly contracting}.  

It is easy to see that, passing over to $\phi^{-1}$, there is a similar way to define the same sets by an analogous property as {\em uniformly $\phi^{-1}$-expanding}  subsets of $\partial^2\FN$.

\end{rem}

\section{steps 3 and 4}
\label{steps-three-and-four}

Throughout this section we assume that $f_-: \tau_- \to \tau_-$ is a train track map that represents the automorphism $\phi^{-1}$ of $\FN$. Since $\phi$ is assumed to be iwip, it follows (as immediate consequence of 
Definition \ref{defn:iwips}) that $\phi^{-1}$ is also iwip.
Just as in the previous section for $f_+: \tau_+ \to \tau_+$, we can assume here,
again after first having possibly replaced $\phi$ by a positive power:

\begin{hyp}
\label{INP-hyp-inv}
Up to inversion there is at most one INP $\eta$ in $\tau_-$, and the two endpoints of $\eta$ are distinct vertices in $\tau_-$.
\end{hyp}

We consider the lamination $L^2(T_+)$ from the previous section, and we shorten here the notation to $L^2_-:= L^2(T_+)$ 
(where the change of sign is natural in that $L^2_-$ contracts under iteration of $\phi$, while $T_+(\phi)$ expands). From 
Proposition \ref{length-contraction} and
Remark \ref{contracting-stable}
we 
obtain 
the following 
strong expansion property:

\begin{lem} 
\label{previous-steps}
For any $\lambda > 1$ there is a constant $c > 0$ and an exponent $t \geq 1$ such that for any pair $(X, Y) \in L^2_-$ and any subpath $\gamma'$ of the geodesic realization $\gamma = \gamma_{\tau_-}(X, Y)$ of $(X, Y)$ in $\tau_-$, with simplicial length $L(\gamma') \geq c$, one has:
$$L([f_-^t(\gamma')]) \geq \lambda L(\gamma')$$
\qed
\end{lem}

\smallskip

For any integer $K \geq 0$ we define the set $\cal L_K = \cal L_K(L^2_-, \tau_-)$ to be the set of subpaths $\gamma$ of length $L(\gamma) = K$ of any biinfinite path $\gamma_{\tau_-}(X, Y)$ in $\tau_-$ which is a geodesic realization any pair $(X, Y) \in L^2_-$. 

For any integer $C \geq 0$ and any path $\gamma$ in $\tau_-$ we denote by $\gamma\chop_C$ the (possibly trivial) subpath of $\gamma$ which is left after erasing from $\gamma$ the two boundary subpaths of simplicial length $C$. 

Recall from Remark \ref{bounded-back} that there is a constant $C(f_-) \geq 0$ which bounds the length of any backtracking subpath in the (unreduced) image $f_-(\gamma)$ of any reduced path $\gamma$ in $\tau_-$.

\begin{rem}
\label{boundary-chopped}
It follows that for any reduced path $\gamma$ in $\tau_-$ and any subpath $\gamma'$ of $\gamma$ the path $[f_-(\gamma')]\chop_{C(f_-)}$ is a subpath of the reduced path $[f_-(\gamma)]$. Note that this statement is not necessarily true if $[f_-(\gamma')]\chop_{C(f_-)}$ is replaced by the path $[f_-(\gamma')]$.
\end{rem}

\begin{lem}
\label{finite-paths}
There exists a constant $C_0 > 0$ such that, for any $C \geq C_0$ and for $t \geq 1$ as in Lemma \ref{previous-steps}, the following holds: 

For any path $\gamma \in \cal L_C$ there exists a path $\gamma' \in \cal L_C$ with the property that $\gamma$ is a subpath of $[f^t_-(\gamma')]\chop_{C(f_-)}$.
\end{lem}

\begin{proof}
This is a direct consequence 
of the above Lemma \ref{previous-steps} 
and of the $f_-$-invariance of $L^2_- = L^2(T_+(\phi))$ (which results directly from $[T_+(\phi)] \phi = [T_+(\phi)]$, see Proposition \ref{prop:LL} and Remark \ref{dual-lamination-projectivized}.
\end{proof}

Recall from Convention \ref{subdivision-INP} that we can assume without loss of generality 
that both endpoints of 
the INP $\eta$ in $\tau_-$ are vertices of $\tau_-$.

\begin{lem}
\label{at-most-1-illegal}
(a) For $C$ as in Lemma \ref{finite-paths} there is at most one illegal turn on any path $\gamma \in \cal L_C$. This illegal turn must be the tip of some 
INP which is contained as subpath in $\gamma$.

\smallskip
\noindent
(b)
Furthermore, if $\gamma \in \cal L_C$ is legal, then there is at most one unused turn on any path $\gamma \in \cal L_C$. If $\gamma$ contains an INP, then all turns outside the INP are used turns. (To be specific: At the initial and terminal point of the INP the path $\gamma$ may well cross over an unused (but legal) turn.)
\end{lem}

\begin{proof}
(a)
For any path $\gamma =: \gamma_0 \in \cal L_C$ we can assume by Lemma \ref{finite-paths} that (after possibly replacing $f$ by a positive power $f^t$) there is an infinite family of paths $\gamma_{n} \in \cal L_C$, for any integer $n \leq 0$, such that 
each $\gamma_n$ is a subpath of $[f_-(\gamma_{n-1})]\chop_{C(\phi)}$. We note (compare Remark \ref{ILT-defn}) that $\ILT(\gamma_n) \geq \ILT(\gamma_m)$, for any $n \leq m \leq 0$. 

Since all $\gamma_n$ have length 
equal to $C$, it follows from Remark \ref{ILT-defn} that for sufficiently negative $n \leq m$ one has $\ILT(\gamma_{n}) = \ILT(\gamma_m)$. 
Thus it follows from Proposition \ref{ILT-decrease}, for $-m$ sufficiently large and any $r \in \N$, that every illegal turn in $\gamma_{m-r}$ 
is the tip of some INP which is entirely contained as subpath in $\gamma_{m-r}$, and that adjacent such INP-subpaths can only overlap in a common endpoint. The same statement must be true for all $f^s(\gamma_{m-r})$ with $s \geq 1$, and hence for all $\gamma_k$ with $k \geq m-r$, and thus also for $\gamma = \gamma_0$.

But the same argument applies to any of the $\gamma_n$.  
It follows that any maximal legal subpath $\gamma'$ of $\gamma_n$, between two adjacent INP-subpaths, is the $f_-$-image of the maximal legal subpath between adjacent INP-subpaths in $\gamma_{n-1}$. But the endpoints of INPs are vertices, so that such non-degenerate maximal legal subpaths can not become arbitrary small.
It follows from the assumption that $f_-$ is expanding that
each of the legal paths $\gamma'$ between two adjacent INP-subpaths of any $\gamma_n$ must has length 0. But by Hypothesis \ref{INP-hyp-inv} there is only one INP in $\tau_-$, and its endpoints are distinct. Thus no reduced path $\gamma$ in $\tau_-$ can contain two subsequent INPs which have a common endpoint.  This shows statement (a).

\medskip
\noindent
(b) 
We first assume that $\gamma$ (and hence any of the $\gamma_n$ as in part (a))  is legal. 
Recall from Remark \ref{rem-used}
(2) that used turns are mapped by $D^2 f_-$ to used turns. However, by Remark \ref{rem-used} (3), any unused legal turn must be mapped similarly to an  unused legal turn, as one can take its (necessarily unused legal) preimage 
via Lemma \ref{finite-paths}
arbitrarily far back.

From the last paragraph it follows that any maximal used legal subpath $\gamma'$ of any of the biinfinite legal paths 
$\gamma_n$, 
(i.e. $\gamma'$ crosses only over used turns but starts and ends at vertices with an unused legal turn on $\gamma_n$), is the $f'$-image of the maximal used legal subpath between adjacent unused legal turns in $\gamma_{n-1}$. But the unused turns do only occur at vertices, so that such non-degenerate maximal legal subpaths can not become arbitrary small.
It follows from the assumption that $f_-$ is expanding that
any of the above used legal paths $\gamma'$ between two adjacent unused turns of any $\gamma_n$ must have length 0; in other words:  there is only one such unused turn on any of the $\gamma_n$.

The same argument applies to the maximal legal subpath between an unused turn and the boundary point of an INP, in case some $\gamma_n$ would contain both. It follows again that this maximal legal subpath must have length 0, so that the only unused turns, on any $\gamma_n$ which contains an INP, can occur at the two endpoints of the INP.  This proves claim (b).
\end{proof}

As a direct consequence of Lemma \ref{at-most-1-illegal} we obtain:

\begin{prop}
\label{L-is-almost-legal}
(a)
For any $(X,Y) \in L^2_-$ there is on any of the geodesic realizations $\gamma := \gamma_{\tau_-}(X,Y)$ in $\tau_-$ at most one ``singularity'': This can either be an unused legal turn, or an illegal turn at the tip of an INP, or an INP with one or two unused legal turns at its boundary points.  The remainder of $\gamma$ is legal and crosses only over used turns.

\smallskip
\noindent
(b)
If $\gamma = \gamma_{\tau_-}(X,Y)$ contains a non-used legal turn, then so does every 
$\gamma_t = \gamma_{\tau_-}(\phi^t(X), \phi^t(Y)) \in L^2_-$, for any $t \in \Z$  (see Remark \ref{no-lift}).
The analogous statement holds if $\gamma$ crosses over an INP with $k$ unused legal turns at its boundary points, for $k \in \{0, 1, 2\}$.
\end{prop}

\begin{proof}
(a)
It suffices to choose $C$ in Lemma \ref{at-most-1-illegal} large enough, so that the two singularities would be both contained in a path of $\cal L_C$, in contradiction to this lemma.

\smallskip
\noindent
(b)
This is a direct consequence of part (a), of the fact that the image of an INP is (after reduction) again an INP, and of the fact (explained in the first paragraph of the proof of part (b) of Lemma \ref{at-most-1-illegal}) that under the map $D^2 f_-$ used turns are mapped to used turns, and any unused legal turn on $\gamma_t$ is mapped to an unused legal turn on $\gamma_{t+1}$.
\end{proof}

In the following we call a (possibly infinite or biinfinite) path $\gamma$ {\em used legal} if it is legal and if it only crosses over used turns.
Recall that such $\gamma$ satisfy $f(\gamma) = [f(\gamma)]$, and that the latter is again used legal.

\begin{cor}
\label{no-singularity}
(a)
For any $(X, Y) \in L^2_-$ every finite used legal subpath $\gamma'$ of the geodesic realization $\gamma_{\tau_-}(X,Y)$ in $\tau_-$ is also contained as subpath in 
$f_-^t(e)$, for some $t \geq 0$ and any edge $e$ of $\tau_-$.

\smallskip
\noindent
(b)
Assume that for $(X,Y) \in L^2_-$
the geodesic realization $\gamma_{\tau_-}(X, Y)$ 
is a used legal path.  Then $(X,Y)$ belongs to the BFH-attracting lamination $L^2_{BFH}(f_-)$ of $f_-:\tau_- \to \tau_-$ (see Definition \ref{BFH-lamination}).
\end{cor}

\begin{proof}
(a)
From the expansiveness of $f_-$ and the legality of  $\gamma'$ we can use Proposition \ref{L-is-almost-legal} (b) to ``iterate'' $f_-$ backwards until we find a path $\gamma''$ which runs over at most two adjacent edges $e, e'$ in $\tau_-$ such that $f_-^t(\gamma'')$ contains $\gamma'$ as subpath, for some $t \geq 0$. But since all turns of $\gamma'$ are used, the same must be true for all of its preimages, so that the turn between $e$ and $e'$ must be used. From Definition \ref{six} it follows directly that this proves the claim (a).

\smallskip
\noindent
(b)
This is a direct consequence of part (a), by the definition of $L^2_{BFH}(f_-)$ in Definition \ref{BFH-lamination}.
\end{proof}

Summing up the results of this section we can state that we know so far that every leaf of $L^2_-$ either belongs to $L^2_{BFH}(f_-)$, or else contains precisely one ``singularity'' as specified in Proposition \ref{L-is-almost-legal} (a).  The goal of the next section will be to show that such ``singular leaves'' belong to the diagonal closure of $L^2_{BFH}(f_-)$, before then concluding with the proof of our main result stated in the Introduction.

\section{steps 5 - 7}
\label{steps-five-to-seven}

Throughout this section we use the same conventions as in section \ref{steps-three-and-four}, in particular Hypothesis \ref{INP-hyp-inv} and the notation from the paragraph preceding it.

\begin{lem}
\label{eigenrays-0}
Let $(\rho_t)_{t \in \Z}$ be a family of infinite reduced rays in $\tau_-$,
and assume that $[f_-(\rho_t)] = \rho_{t+1}$ for all $t \in \Z$. Assume furthermore that each $\rho_t$ is legal, and that each $\rho_t$ starts at a vertex $v_t$ of $\tau_-$.
Then each $\rho_t$ is an eigenray for the train track map $f_-: \tau_- \to \tau_-$.
\end{lem}

\begin{proof}
Since each $\rho_t$ is legal, one has always $f_-(\rho_t) = [f_-(\rho_t)] = \rho_{t+1}$. It follows that $f_-$ maps the starting point $v_t$ of $\rho_t$ to the starting point $v_{t+1}$ of $\rho_{t+1}$.
Since the starting point of each $\rho_t$ is (by hypothesis) a vertex, it follows from the finiteness of the vertex set of $\tau_-$ that each $v_t$ is a periodic vertex of $\tau_-$. 

Similarly, since all $\rho_t$ are legal, 
the initial edge $e_{t+1}$ of $\rho_{t+1}$ must always be equal to the initial edge of the (legal) edge path $f_-(e_t)$, where $e_t$ denotes the initial edge of $\rho_t$.
Again by the finiteness of the graph $\tau_-$, we see that for any $t \in \Z$ the initial edge $e_t$ of $\rho_t$ lies on a periodic orbit, with respect to the map $D f_-$ (see Definition \ref{expanding-and-D} (b)).
In other words: each $e_t$ is the eigen edge of the gate to which it belongs (compare subsection \ref{rays}).

We now argue by induction: 
Suppose that all $\rho_{t+k}$ for some fixed integer $k \neq 0$ start with the same initial subpath $\gamma_n$ of simplicial length $n \geq 1$. Then it follows from the expansiveness of $f_-$ (and from the fact that all paths are legal) that their images have a common initial subpath of simplicial length $\geq n+1$. Thus the assumption $f_-^k(\rho_t) = \rho_{t+k}$ implies that the maximal common initial subpath of the $\rho_{t+k}$ must have infinite length, so that we obtain $\rho_t =\rho_{t+k}$ for all $t \in \Z$. But this is the defining equality for eigenrays.
${}^{}$
\end{proof}

\begin{lem}
\label{prolongation+}
(a) For each eigenray $\rho$ which starts at a vertex of $\tau_-$ there is an eigenray $\rho'$, with same initial vertex, such that $\bar \rho \circ \rho'$ is a biinfinite used legal path.

\smallskip
\noindent
(b)
Any path of type $\bar \rho \circ \rho'$, with used legal turn at the concatenation point, realizes a pair $(X, Y) \in 
\partial^2 \FN$ which belongs to the BFH-attracting lamination $L^2_{BFH}(f_-)$:
$$
\bar \rho \circ \rho' = \gamma_{\tau_-}(X, Y) \quad {\rm with} \quad (X, Y) \in L^2_{BFH}(f_-)$$

\end{lem}

\begin{proof}
(a)
We first derive from Lemma \ref{prolongation} that there is an edge $e$ with same initial vertex $v$ as $\rho$ such that $\bar \rho \circ e$ is used legal. In particular $e$ doesn't lie in the gate from which $\rho$ starts. By Proposition \ref{gates-eigenrays} there is an eigenray $\rho'$ which starts from the gate that contains $e$, and by Lemma \ref{used-gate-turns} the turn from $\bar \rho$ to $\rho'$ is used. Furthermore, every eigenray is used legal, since every finite subpath is contained in some $f_-$-iterate of its initial edge (compare Proposition \ref{gates-eigenrays} and its proof). This proves claim (a).

\smallskip
\noindent
(b)
By 
hypothesis 
the turn at the concatenation point between $\bar \rho$ and $\rho'$ is used: There is an edge $e$ of $\tau_-$ and an exponent $t \geq 1$ such that the edge path $f^t_-(e)$ contains a subpath 
$\gamma$ which 
is a concatenation of non-trivial initial segments of both, $\bar \rho$ and $\rho'$. It follows from the expansiveness of $f_-$ and the definition of an eigenray that every 
finite subpath of $\bar \rho \circ \rho'$ is also a subpath of some $f$-iterate of $\gamma$, and hence of
$f^{nt}_-(e)$ for some integer $n \geq 0$. By Definiton \ref{BFH-lamination} it follows that $(X,Y)$ is an element $L^2_{BFH}(f_-)$.
\end{proof}

Recall 
that for any $(X, Y) \in \partial^2 \FN$ the geodesic realization in $\tau_-$ satisfies $[f_-(\gamma_{\tau_-}(X, Y))] = \gamma_{\tau_-}(\phi^{-1} (X), \phi^{-1}(Y))$,
see Remark \ref{no-lift}.
As in section \ref{steps-three-and-four} we 
will use below 
for the lamination dual to the forward limit tree $T_+$ the abbreviation $L^2_- = L^2(T_+)$.

\begin{prop}
\label{eigenrays-1}
Let $(X, Y) \in \partial^2\FN$, and let $(\gamma_t)_{t \in \Z}$ be the family of biinfinite reduced paths $\gamma_t = \gamma_{\tau_-}( \phi^{-t} (X),  \phi^{-t}(Y))$ which realize in $\tau_-$ the 
$\phi$-orbit of $(X, Y)$.
In particular we assume that $[f_-(\gamma_t)] = \gamma_{t+1}$ for all $t \in \Z$. 

\smallskip
\noindent
(a)
Assume that each $\gamma_t$ is legal, and that on each $\gamma_t$ there is precisely one unused legal turn.
Then there are 
eigenrays $\rho_t$ and $\rho'_t$, such that $\gamma_t = \bar \rho_t \circ \rho'_t$, where the unused turn is situated at the concatenation point. 

\smallskip
\noindent
(b)
Assume that on each $\gamma_t$ there is precisely one INP subpath $\eta_t$, and that every turn of $\gamma_t$ is legal except for the turn at the tip of $\eta_t$.
Then there are 
eigenrays $\rho_t$ and $\rho'_t$, such that $\gamma_t = \bar \rho_t \circ \eta_t\circ \rho'_t$. 
\end{prop}

\begin{proof}
Note that in case (a) each unused turn in $\gamma_t$ takes place at a vertex $v_0^t$ of $\tau_-$. Similarly, in case (b) the endpoints $v_1^t$ and $v_2^t$ of any INP in $\gamma_t$ are vertices, by Convention \ref{subdivision-INP}. 
Notice also that in case (a) one has $[f_-(\gamma_t)] = f_-(\gamma_t)$, and in case (b) the only backtracking subpath of $f_-(\gamma_t)$ is situated at the tip of $\eta_{t+1}$. 
Hence
$f_-$ maps the ``singular'' vertices $v_i^t$ in $\gamma_t$ to the singular vertices $v_i^{t+1}$ in $\gamma_{t+1}$. Thus we obtain the statements of (a) and (b) is a direct consequence of Lemma \ref{eigenrays-0}.
\end{proof}

\begin{defn}
\label{equivalence-relation}
For any two eigenrays $\rho, \rho'$ in $\tau_-$  we write
$$\rho \sim \rho'$$
if $\rho$ and $\rho'$ have as initial points a common vertex $v$ of $\tau_-$, if $v$ is $f_-$-periodic, and 
if in the biinfinite path $\bar \rho \circ \rho'$ the turn at the concatenation point is used. 
Thus the relation $\sim$ generates an
equivalence relation (also denoted by $\sim$) on the finite set of eigenrays in $\tau_-$ which start at periodic vertices of $\tau_-$.
\end{defn}

It follows directly 
from the definitions that the equivalence relation $\sim$ is modeled precisely after the definition of the diagonal extension (see Definition-Remark \ref{diagonal-closure}), so that one observes:

\begin{rem}
\label{diagonal-equivalence}
Let $v$ be a periodic vertex of $\tau_-$, and assume that the BFH-attracting lamination $L^2_{BFH}(f_-)$ contains pairs $(X_1, Y_1)$ and $(X_2, Y_2)$ 
with $X_1 \neq X_2$. Assume that the geodesic realizations $\gamma_1 = \gamma_{\tau_-}(X_1, Y_1)$ and $\gamma_2 = \gamma_{\tau_-}(X_2, Y_2)$ are both concatenations of eigenrays, $\gamma_1 = \bar \rho_1 \circ \rho'_1$ and $\gamma_2 = \bar \rho_2 \circ \rho'_2$, with concatenation vertices both equal to $v$.  Then $\rho_1 \sim\rho_2$ implies that $(X_1, X_2)$ is contained in the diagonal extension of $L^2_{BFH}(f_-)$.

Note that the converse implication (not used here) is also true, but it is a little bit less immediate: one needs to employ Lemma \ref{eigenrays-0} and Lemma \ref{prolongation+} (b).
\end{rem}

We will now use the equivalence relation $\sim$ to derive from $\tau_-$ and $f_-$ a new graph $\tau_-^*$ and a new self-map $f_-^*$ 
(compare in this context \cite{Lu2}). This will be defined in several steps as follows:

We assume that at some periodic vertex $v$ of $\tau_-$ there is more than one $\sim$-equivalence class of eigenrays. We first assume that $v$ is fixed by $f_-$.

We then split the vertex into several new vertices, called {\em subvertices} of $v$, precisely one for each equivalence class, such that all eigenrays starting at the same subvertex are equivalent, and conversely. 

We now attach any other edge $e$ of $\tau_-$ with initial vertex $v$ to one of the new subvertices, according to which eigenray started in $\tau_-$ from the same gate as $e$. 
Recall from Proposition \ref{gates-eigenrays} that for each gate $\frak g_i$ at a periodic vertex there is precisely one eigenray starting from $\frak g_i$.

We connect the $k \geq 2$ subvertices of $v$ by a $k$-pod $\cal P$ to obtain a new graph which transforms back to $\tau_-$, if $\cal P$ is contracted to a single point. 

We now consider any vertex $v'$ of $\tau_-$ which is mapped eventually by $f_-$ to $v$, and we do the analogous construction there, where the subdivision of $v'$ into subvertices is simply lifted from $v$ to $v'$. (Indeed, since $v'$ is not periodic, we can not use again the relation $\sim$ for the subdivision of $v'$.) The graph obtained 
by this {\em blow-up} of vertices is denoted by $\tau_-^*$ 
and we denote by $\cal P^*$ the union of all $k$-pods introduced at $v$ and at its preimage vertices.

We observe that the train track map $f_-$ induces naturally a map $f_-^*: \tau_-^* \to \tau_-^*$ which maps $\cal P^*$ 
to itself and leaves also $\tau_-^* \smallsetminus \overset{\circ}{\cal P^*}$ invariant. Since $f_-$ is expanding, it follows that $\tau_-^* \smallsetminus \overset{\circ}{\cal P}$ has non-trivial fundamental group. Hence $\pi_1(\tau_-^* \smallsetminus \overset{\circ}{\cal P})$ is a non-trivial proper $\phi$-invariant free factor of $\FN$, contradicting the assumption that $f_-$ represents an iwip automorphisms.

\smallskip

Now, if $v$ is not fixed by $f_-$, one has to do the same blow-up procedure at the whole $f_-$-orbit of $v$, but the argument remains in this case precisely the same as in the simpler case just considered.

This shows:

\begin{lem}
\label{only-one-equi-class}
At every periodic vertex $v$ of $\tau_-$ there is only one $\sim$-equivalence class of eigenrays.
\qed
\end{lem}

\begin{prop}
\label{every-singular-leaf-is-diagonal}
Let $(X,Y) \in L^2_-$ be such that its geodesic realization $\gamma_{\tau_-}(X,Y)$ in $\tau_-$ is not used legal.  Then $(X, Y)$ belongs to the diagonal extension of the BFH-attracting sublamination $L_{BFH}^2(f_-)$ of the train track map $f_-$.
\end{prop}

\begin{proof}
From  
Proposition \ref{L-is-almost-legal} 
and the hypothesis that $(X, Y) \in L^2_-$ but $\gamma := \gamma_{\tau_-}(X,Y)$ not used legal
we know that $\gamma$
(and hence any geodesic realization $\gamma_t
= \gamma_{\tau_-}(\phi^t(X),\phi^t(Y))$, for  arbitrary $t \in \Z$) contains either precisely one non-used legal turn, or else it runs precisely once over an INP-subpath  $\eta$ of $\tau_-$. Thus one deduces from Proposition \ref{eigenrays-1} that $\gamma = \bar \rho \circ \rho'$ or $\gamma = \bar \rho \circ \eta\circ \rho'$, for eigenrays $\rho$ and $\rho'$. 

In the first case the unused turn occurs at the concatenation vertex $v$ of the two eigenrays, and it follows directly from Lemma \ref{prolongation+} that there are eigenrays $\rho_1$ and $\rho_2$ at $v$ such that the biinfinite paths $\gamma_1 = \bar \rho \circ \rho_1$ and $\gamma_2 = \bar \rho' \circ \rho_2$ are used legal, and that they are geodesic realizations $\gamma_1 = \gamma_{\tau_-}(X, Z_1)$ and $\gamma_2 = \gamma_{\tau_-}(Y, Z_2)$ of pairs $(X, Z_1)$ and $(Y, Z_2)$ that both belong to to the BFH-attracting lamination $L^2_{BFH}(f_-)$. Hence it follows directly from Lemma \ref{only-one-equi-class} and Remark \ref{diagonal-equivalence} that $(X, Y)$ belongs to the diagonal extension of $L^2_{BFH}(f_-)$.

In the second case, where $\gamma = \bar \rho \circ \eta\circ \rho'$, we consider the eigenrays $\rho_1$ and $\rho_2$ which have the two legal branches of the INP $\eta$ as initial subpaths and agree along an infinite legal subray, see Remark \ref{train-track-maps}. By the case treated before, the biinfinite paths $\bar \rho \circ \rho_1$ and $\bar \rho_2 \circ \rho'$ realize pairs $(X,Z)$ and $(Z, Y)$ which belong to the diagonal extension of $L^2_{BFH}$. Thus $(X,Y)$ also belongs to this diagonal extension.
\end{proof}

We have now assembled all ingredients necessary to prove the main result of this paper:

\begin{proof}[Proof of Theorem \ref{main-result2}]
We first note that we may well replace $\phi$ by any of its positive powers, since for any integer $k \geq 1$ we have both, $T_+(\phi^k) = T_+(\phi)$ (see Remark \ref{powers-trees}) and $L^2_{BFH}(f_-^k) = L^2_{BFH}(f_-)$  (see Proposition \ref{from-BFH} (3)). Hence we can apply our 
previous 
results from sections \ref{steps-one-and-two}, \ref{steps-three-and-four} and \ref{steps-five-to-seven}.

In Corollary \ref{no-singularity} (b) together with Proposition \ref{every-singular-leaf-is-diagonal} it is shown that $L^2(T_+)$ is contained in the diagonal extension $\diag(L^2_{BFH}(f_-))$. More precisely, a pair $(X, Y) \in L^2(T_+)$ is contained in $\diag(L^2_{BFH}(f_-)) \smallsetminus L^2_{BFH}(f_-)$ if and only if its geodesic realization in $\tau_-$ is a concatenation of eigenrays at either an unused legal turn, or else at an INP. But since eigenrays in $\tau_-$ are in 1-1 relation with periodic gates of train track map $f_-: \tau_- \to \tau_-$ (see Corollary \ref{eigen-bijection}), there exist only finitely many distinct eigenrays. Since also the number of unused legal turns is finite, 
and up to inversion there is only one INP in $\tau_-$,
it follows that $L^2(T_+)$ contains only finitely many $\FN$-orbits of pairs $(X,Y)$ which are contained in $\diag(L^2_{BFH}(f_-)) \smallsetminus L^2_{BFH}(f_-)$. This establishes claim (4) of Theorem \ref{main-result2}.

Furthermore, since we know from the previous paragraph that any sublamination of $L^2(T_+) \subset \diag(L^2_{BFH}(f_-))$ either contains some leaf of $L^2_{BFH}(f_-)$ or else it must contain $L^2(\rho)$ for some eigenray $\rho$, which by Lemma \ref{eigenray-BFH} is contained in $L^2_{BFH}(f_-)$. Thus every minimal sublamination of $L^2(T_+)$ must non-trivially intersect $L^2_{BFH}(f_-)$, which is itself minimal, see Proposition \ref{from-BFH} (2).
Hence $L^2_{BFH}(f_-)$ 
must be contained in $L^2(T_+)$, as unique minimal sublamination.
This proves parts (1) and (2) of Theorem \ref{main-result2}. 

By Proposition \ref{dual-diag-closed} the lamination $L^2(T_+)$ is diagonally closed, so that it contains with $L^2_{BFH}(f_-)$ also its 
diagonal extension and its diagonal closure:  $\diag(L^2_{BFH}(f_-)) \subset \diagclos(L^2_{BFH}(f_-)) \subset L^2(T_+)$. These inclusions, together with the above proved converse inclusion 
$L^2(T_+) \subset \diag(L^2_{BFH}(f_-))$, prove the equalities claimed in part (3) of Theorem \ref{main-result2}.
\end{proof}

\begin{rem}
\label{last}
The reader should notice that what we have shown in the last proof is stronger than what is claimed the statement of Theorem \ref{main-result2} in the Introduction, in the following sense:  In the proof and in the preceding material of this paper we never used statement (1) of Proposition \ref{from-BFH}. Thus what we have shown above is the statement of Theorem \ref{main-result2}, but with $L^2_{BFH}$ replaced by $L^2_{BFH}(f_-)$ for any train track representative $f_-$ of a suitable power $\phi^{-k}$ as stated in the beginning of the above proof.  In particular, this shows (by the uniqueness property (2) of Theorem \ref{main-result2}) that the lamination $L^2_{BFH}(f_-)$ does not depend on the particular choice of the train track map $f_-$.
\end{rem}

\section{Discussion}
\label{discussion}

Throughout this section $T$ denotes always an $\R$-tree from $\cvnbar$, and $\mu$ a current from $\curr$. Recall that such a pair $(T, \mu)$ is called 
{\em perpendicular} if $\langle T, \mu \rangle = 0$. Recall also that following \cite{KL3} this is equivalent to the statement $\supp(\mu) \subset L^2(T)$, which in turn is equivalent to $\diagclos(\supp(\mu)) \subset L^2(T)$.  The purpose of this section is to give (partial) answers to the following question:

{\em Under which circumstances does this last inclusion actually improve to}
$$\diagclos(\supp(\mu)) = L^2(T) \, ?$$

\smallskip

To shape the discussion a bit, we propose:

\begin{defn}
\label{diagonally-equal}
(a)
A pair $(T, \mu) \in \cvnbar \times \curr$ is said to be {\em diagonally equal} if $L^2(T) = \diagclos(\supp(\mu))$. In this case we also say that {\em $T$ is diagonally equal to $\mu$}, or {\em $\mu$ is diagonally equal to $T$}.

\smallskip
\noindent
(b)
A tree $T$ is called {\em diagonally equalizable (DE)} if there exists a current $\mu$ such that $(T, \mu)$ is diagonally equal. Similarly, a current $\mu$ is called {\em diagonally equalizable (DE)} if there exists a tree $T$ such that $(T, \mu)$ is diagonally equal.

\smallskip
\noindent
(c)
A tree $T$ is called {\em totally diagonally equalizable (TDE)} if every perpendicular current $\mu$ is diagonally equal to $T$. Similarly, a current $\mu$ is called {\em totally diagonally equalizable (TDE)} if every perpendicular tree $T$ is diagonally equal to $\mu$.

\smallskip
\noindent
(d)
A tree $T$ is called {\em uniquely diagonally equalizable (UDE)} if among all perpendicular currents there is (up to scalar multiples) precisely one current $\mu$ which is diagonally equal to $T$. Similarly, a current $\mu$ is called {\em uniquely diagonally equalizable (UDE)} if among all perpendicular trees there is (up to scalar multiples) precisely one tree $T$ which is diagonally equal to $\mu$.
\end{defn}

All of these definitions descend directly to the projectivized objects, so that one can speak for example of a ``uniquely diagonally equalizable'' $[T] \in \CVNbar$.  The result of Theorem \ref{main-thm} can be rephrased by stating that for any atoroidal iwip $\phi \in \Out(\FN)$ the attracting fixed point $[T_+] \in \CVNbar$ and the repelling fixed point $[\mu_-] \in \PCurr$ constitute a pair which is diagonally equal.

\smallskip

The reader may want to note that any tree $T$ which is both, TDE and UDE, must be dually uniquely ergodic (for the terminology see 
\cite{CHL1-III}). 
Similarly, if $T$ is perpendicular to any current $\mu$ which is both, TDE and UDE, then $T$ must be uniquely ergometric (as defined in \cite{CHL2}).
We will see below that for trees (but not necessarily for currents) UDE implies TDE.

\smallskip

By the above stated result of \cite{KL3} we know that every diagonally equal pair $([T],[\mu])$ must be perpendicular. The set $I_0$ of such perpendicular pairs contains the uniquely determined minimal set $\cal M^2 \subset \CVNbar \times \PCurr$ with respect to the $\Out(\FN)$-action (see \cite{KL1}), which is also contained in the cartesian product $\cal M^{cv} \times \cal M^{curr}$ of the (much better understood) minimal sets  $\cal M^{cv} \subset \CVNbar$ and $\cal M^{curr} \subset \PCurr$.
However, we only know these inclusions; it is open whether $\cal M^2$ is equal to $I_0 \cap (\cal M^{cv} \times \cal M^{curr})$.

\begin{rem}
\label{subminimal}
We will give below examples of perpendicular pairs which are not diagonally equal. The set $\cal{DE}$ of diagonally equal pairs $([T], [\mu])$ is by definition $\Out(\FN)$-invariant and non-empty, but it will follow from the results presented below that $\cal{DE}$ is not closed in $\CVNbar \times \PCurr$.
\end{rem}

\begin{rem}
\label{subintersectiongraph}
The set $\cal{DE}$ defines also an $\Out(\FN)$-invariant subgraph of the intersection graph $\cal I(\FN)$ defined in \cite{KL2}, and it seems worthwhile to think which kind of a subgraph this is.  Since all ``limit pairs'' $([T_+], [\mu_-])$ for atoroidal iwips are contained in $\cal{DE}$, it must consist of many distinct connected components. It seems likely that even the subgraph $\cal{DE}_0$ defined by all $\cal{DE}$-pairs which are contained in the main component $\cal I_0(\FN)$ of $\cal I(\FN)$ is non-connected.
\end{rem}

We will now turn to the question of necessary and sufficient conditions for a pair $(T, \mu)$ to be diagonally equal, and for $T$ or $\mu$ to be DE, TDE or UDE. 
We first recall some facts:

\begin{facts}
\label{known-facts}
(1)
For every algebraic lamination $L^2$ over $\FN$ there exist a current $\mu \neq 0$ with $\supp(\mu) \subset L^2$ (see \cite{CHL1-III}).  In particular, for every tree $T \in \partial \cvn := \cvnbar \smallsetminus \cvn$ there exists a current $\mu \in \curr$ which is perpendicular to $T$.

The dual statement is wrong, if we accept the subspace 
$\Curr_{+}(\FN) \subset \curr$ of currents $\mu \in \curr$ with {\em full support} (i.e. $\supp(\mu) = \partial^2 \FN$) as ``dual'' of $\cvn$. (Note that 
this dualization is natural in that
$\cvn$ can be characterized also as the subspace in $\cvnbar$ which consists of all $\R$-trees $T$ with ``empty zero lamination'').  Indeed, there exist currents $\mu \in \curr \smallsetminus \Curr_{+}(\FN)$ which are not perpendicular to any tree in $\cvnbar$. Interesting examples of such {\em filling} currents are given in \cite{KL4}.

\smallskip
\noindent
(2)
If $T, T' \in \partial \cvn$, and if there is a length decreasing $\FN$-equivariant map $T \to T'$, then one has $L^2(T) \subset L^2(T')$. This applies in particular to the case where $T'$ results from $T$ by contracting an $\FN$-invariant forest $T_0$:  Here each connected component of $T_0$ is mapped to a distinct point in $T'$. If $T_0$ contains connected components with infinite diameter, then the inclusion $L^2(T) \subset L^2(T')$ is strict: one has $L^2(T) \neq L^2(T')$.

\smallskip
\noindent
(3)
Let $f: \tau \to \tau$ be an expanding train track map which represents some $\psi \in \Out(\FN)$, but contrary to the case considered in the main body of this paper, we assume now that 
there is an $f$-invariant {\em bottom stratum} $\tau_{bottom}$ of $\tau$ such that $\pi_1(\tau_{bottom})$ is a $\psi$-invariant non-trivial proper free factor of $\FN$:  
hence $\psi$ is not iwip.

We assume further that 
the non-negative transition matrix $M(f)$ 
(assumed to be, after transposition, in normal form for non-negative matrices) has precisely two primitive diagonal submatrices, with eigenvalues $\lambda_{top}$ and $\lambda_{bottom}$ respectively, and that the lower left off-diagonal block of $M(f)$ is non-zero (while, according to the normal form for the transpose of $M(f)$, the upper right off-diagonal block must be the zero matrix).  For such expanding {\em two-strata} train track maps one has to distinguish two cases: 

If $\lambda_{top} < \lambda_{bottom}$, then both eigenvalues possess non-negative row eigenvectors $\vec v^{\,*}_{top}$ and $\vec v^{\,*}_{bottom}$ respectively, and each of them determines a {\em Perron-Frobenius-tree} $T_{top} \in \partial \cvn $ and $T_{bottom} \in \partial \cvn $ respectively. Both projective classes, $[T_{top}] \in \partial \CVN $ and $[T_{bottom}] \in \partial \CVN$ are fixed points of the induced action of the automorphism $\psi \in \Out(\FN)$ which is represented by the train track map $f$, see \cite{Lu2}, and they are the only such \cite{Lu1}.

It is easy to see that the $\FN$-action on $T_{top}$ is never free, as all of the bottom stratum acts elliptically.  On the other hand, $T_{bottom}$ may well have free $\FN$-action.

If $\lambda_{top} \geq \lambda_{bottom}$, then only $\vec v^{\, *}_{top}$ is non-negative, so that only $T_{top} \in \partial \cvn $ exists (and is projectively fixed by $\psi$), but no $T_{bottom}$ as in the other case.

On the other hand, if $\lambda_{top} > \lambda_{bottom}$, then there are two non-negative column eigenvectors $\vec v_{top}$ and $\vec v_{bottom}$ respectively, and both determine projectively $\psi$-invariant 
{\em Perron-Frobenius} 
currents $\mu_{top}, \mu_{bottom} \in \curr$, see 
\cite{BHL}.
If $\lambda_{top} \leq \lambda_{bottom}$, then (up to scaling) only one non-negative column eigenvector $\vec v_{bottom}$ exists,  and hence one has only one projectively $\psi$-invariant current $\mu_{bottom} \in \curr$, see \cite{BHL}.

\smallskip
\noindent
(4)
We have to carry the example of a 2-strata automorphism $\psi \in \Out(\FN)$ considered in part (3) above one step further: It is known that the inverse $\psi^{-1}$ is also a 2-strata automorphism, 
with respect to any (possibly relative) train track representative $f_-$ of $\psi^{-1}$.
The bottom stratum for both train track representatives, $f$ (as in (3)) and $f_-$ must determine (up to conjugation) the same free factor of $\FN$, if the restriction of $\psi$ to this factor is iwip.

In any case, the bottom stratum of $f_-$ determines a $\psi$-invariant Perron-Frobenius current $\mu_{bottom}^-$ which is perpendicular to both, 
$T_{top}$ and $T_{bottom}$
(if the latter exists), by a standard argument using the $\psi$-invariance of the intersection form $\langle T_{top}, \mu^-_{bottom} \rangle$ (or $\langle T_{bottom}, \mu_{bottom}^- \rangle$) and the fact that the eigenvalues for those trees are positive for $\psi$ while that of $\mu_{bottom}^-$ is negative (again for $\psi$).

By modifying $\psi$ through iterating only its restriction on the bottom part one can always achieve for the growth rates on the two strata that $\lambda_{top} < \lambda_{bottom}$ for $\psi$ and also $\lambda^-_{top} < \lambda^-_{bottom}$ for $\psi^{-1}$.

\end{facts}

It turns out that for trees one has more tools than for currents, so that we obtain:

\begin{prop}
\label{answers}
Using the notations defined in Facts \ref{known-facts} (3) and (4) above, we have:

\begin{enumerate}
\item
Not every $T \in \partial \cvn$ is DE. Counterexamples are given by the trees $T_{bottom}$, for any 2-strata exponential train track map $f$ 
with 
$\lambda_{top} < \lambda_{bottom}$, 
$\lambda^-_{top} < \lambda_{bottom}^-$ and 
without INP's.
We also assume that the automorphism $\psi \in \Out(\FN)$ represented by $f$ restricts to an atoroidal iwip on the bottom stratum of $f$.

\item
A tree $T \in \partial \cvn$ is TDE if and only if the dual lamination $L^2(T)$ is the diagonal closure of a (uniquely determined) minimal sublamination $L^2_0(T)$.
\item
Every tree $T \in \partial \cvn$ which is UDE must also TDE. A TDE-tree $T$ is UDE if and only if the minimal sublamination $L^2_0(T) \subset L^2(T)$ is uniquely ergodic.
\end{enumerate}
\end{prop}

\begin{proof}
(1)
[sketched only, details will be given elsewhere]
We first apply the main result of \cite{BHL} to some (possibly relative) train track representative $f_-$ of $\psi^{-1}$, which gives a bijection between the set of projectivized $\psi^{-1}$-invariant 
expanding currents and the set of non-negative column eigenvectors (up to rescaling) of $M(f_-)$.
Hence it follows from our hypothesis on the eigenvalues of $M(f)$ that $\mu^-_{bottom}$ is (up to scalar multiples) the only $\psi$-invariant 
expanding current. 

In 
\cite{CH}
it has been shown that for every $T \in \cvnbar$ with dense orbits and free $\FN$-action the number of projectivized ergodic currents $[\mu_i]$ carried by $L^2(T)$ is finite, so that $[T] \psi = [T]$ implies that (after possibly replacing $\psi$ by a suitable positive power) $\psi$ fixes every $[\mu_i]$ in this finite set. The expansiveness of $f$ guarantees that $T_{bottom}$ has dense orbits, the INP-assumption gives the free $\FN$-action on $T_{bottom}$, and $[T_{bottom}] \psi = [T_{bottom}]$ is a direct consequence of the fact that $T_{bottom}$ is a Perron-Frobenius tree for the train track map $f$.  

Since the leaves of $L^2(T_{bottom})$ are uniformly contracting under the action of $\psi$, it follows that every projectively $\psi$-invariant current $\mu$ carried by $L^2(T_{bottom})$ must be $\psi^{-1}$-expanding.  But in the first paragraph of this proof we have shown that $\mu^-_{bottom}$ is up to rescaling the only such $\psi^{-1}$-invariant expanding current. Since every current carried by $L^2(T_{bottom})$ is a linear combination of the extremal (= ergodic) ones, it follows that $[\mu^-_{bottom}]$ is the only projectivized current carried by $L^2(T_{bottom})$.

But $\mu_{bottom}^-$ is carried by the bottom stratum, so that the diagonal closure $\diagclos(\mu_{bottom}^-)$ is also carried by the bottom stratum, while $L^2(T_{bottom})$ must contain all of the BFH-attracting lamination for $\phi^{-1}$, which also contains paths that cross through the top stratum. Hence we have $\diagclos(\mu_{bottom}^-) \neq L^2(T_{bottom})$, which proves the claim.

\smallskip
\noindent
(2)
The ``if'' direction is immediate. In order to prove the ``only if'' direction we consider a minimal sublamination $L'^2$ which has strictly smaller diagonal closure than $L^2(T_{})$.  By Fact \ref{known-facts} (1) there is a current $\mu'$ with support in $L'^2$, which hence can not be diagonally equal to $T_{}$.

\smallskip
\noindent
(3)
If $T$ is not TDE, then by (2) in $L^2(T_{})$ there is a sublamination $L'^2$ with strictly smaller diagonal closure than $L^2(T_{})$, which by Facts \ref{known-facts} (1) carries some current $\mu'$. In particular, $\mu'$ is different from any current $\mu$ which is diagonally equal to $T_{}$. Hence with any such $\mu$ every $\mu + \epsilon \mu'$ is also diagonally equal to $T_{}$, for any $\epsilon > 0$.

The statement about unique ergodicity is really only a definitory statement.
\end{proof}

For the dual setting we get analogous, but in part weaker statements:

\begin{prop}
\label{current-answers}
Let $\mu \in \curr \smallsetminus \curr_+$, and assume that $\mu$ is not filling.
\begin{enumerate}
\item
There exist such $\mu$ which are not DE. Examples
are given by the Perron-Frobenius current $\mu_{top}$, for any 2-strata exponential train track map with $\lambda_{top} > \lambda_{bottom}$ and $\lambda^-_{top} > \lambda^-_{bottom}$.
\item
A current $\mu$ is TDE if 
its support $\supp(\mu)$ has diagonal closure which is within the set of arborescent laminations maximal 
with respect to the inclusion
($\diag(\supp(\mu))$ is ``maximal'').
\item
A TDE-current $\mu$ with maximal diagonal closure of $\supp(\mu)$ is UDE if and only if $\diag(\supp(\mu))$ is uniquely ergometric.
\end{enumerate}
\end{prop}

\begin{proof}
The proof is given by a dualization of the above proof of the corresponding parts of Proposition \ref{answers}. 
The proof of part (1) becomes at some places easier, since $\CVNbar$ is finite dimensional.
\end{proof}



\end{document}